\documentclass[a4paper,10pt]{article}

\usepackage{amsthm}
\usepackage{amsmath}
\usepackage{amssymb}
\usepackage{amsfonts}
\usepackage{enumerate}
\usepackage{authblk}
\usepackage{fullpage}
\usepackage{hyperref}

\newcommand\CC{\mathbf{C}}
\newcommand\RR{\mathbf{R}}
\newcommand\QQ{\mathbf{Q}}
\newcommand\ZZ{\mathbf{Z}}
\newcommand\NN{\mathbf{N}}

\newcommand\AAA{\mathbf{A}}
\newcommand\sign{\mathrm{sign}}
\newcommand\whit{\mathcal{W}}
\newcommand\fin{\mathrm{fin}}
\newcommand\J{\mathbf{J}}
\newcommand\M{\mathbf{M}}
\newcommand\weyl{\mathrm{w}}

\newcommand\Gammatr{\Gamma(\frak{a}\rightarrow\frak{ab}^2)}
\newcommand\const{\mathrm{const.}}
\newcommand\FS{\mathrm{FS}}

\newcommand\lfact[2]{#1\backslash #2}
\newcommand\rfact[2]{#1 / #2}
\newcommand\lrfact[3]{#1\backslash #2 / #3}
\newcommand\GL{\mathrm{GL}}
\newcommand\GLone[1]{\mathrm{GL}_{#1}}
\newcommand\GLtwo[2]{\mathrm{GL}_{#1}(#2)}

\newcommand\SLone[1]{\mathrm{SL}_1}
\newcommand\SLtwo[2]{\mathrm{SL}_{#1}(#2)}
\newcommand\norm[1]{\mathcal{N}(#1)}

\theoremstyle{definition}
\newtheorem{defi}{Definition}

\theoremstyle{plain}
\newtheorem{lemm}{Lemma}
\newtheorem{theo}{Theorem}

\begin{document}

\title{A semi-adelic Kuznetsov formula over number fields}
\author{P\'eter Maga\footnote{MTA Alfr\'ed R\'enyi Institute of Mathematics, Budapest; Central European University, Budapest}}
\date{}
\maketitle

\begin{abstract}
In these notes, we prove a semi-adelic version of the Kuznetsov formula over arbitrary number fields. The extent is the set of those automorphic vectors which are not necessarily spherical in the archimedean aspect and a class of weight functions which is important for applications.
\end{abstract}

\section{Introduction and notation}

In the theory of automorphic forms, Kuznetsov's formula is among the most important and most frequently used tools. Briefly speaking, it matches a certain weighted sum of products of Fourier coefficients (or equivalently, Hecke eigenvalues) of an orthogonal basis in the cuspidal space with a sum of Kloosterman sums, weighted by a Bessel transform.

Kuznetsov \cite{Kuznetsov} originally proved his formula for the modular group over the rational field $\QQ$, which was an extension of the Petersson trace formula (which refers to holomorphic forms) to weight $0$ Maass forms.

Since then, many generalizations and reformulations were born. For totally real number fields, see the work of Bruggeman and Miatello \cite{BruggemanMiatello}, which includes the principal series representations and the discrete series representations in a single formula. For general number fields, the thesis of Venkatesh \cite{Venkatesh} gives a treatment to spherical vectors with references to \cite{BruggemanMiatello2}.

In this paper, we work out a formula to the more general case: non-spherical vectors are also included. In the derivation, we follow \cite{Venkatesh}, borrowing the archimedean investigations from \cite{BruggemanMiatello}, \cite{BruggemanMotohashi} and \cite{Lokvenec}. In the first three sections, we introduce the objects we shall need later. Then at the end of Section 3, we formulate Theorem \ref{kuznetsov}, the main result of this paper. In Sections 4-6, we define the Poincar\'e series, and compute their inner product in two ways. In Section 7, we derive Theorem \ref{kuznetsov}. A technical detail, namely, the convergence of the Poincar\'e series, is worked out in the Appendix.

As an application, we mention that the Kuznetov formula is an important ingredient in the result of \cite{BlomerHarcos}: a Burgess type subconvex bound for twisted $\GLone{2}$ $L$-functions over totally real fields. The project of the author to extend their method to any number field is in progress. We note that recently Wu \cite{Wu} proved this general Burgess type subconvexity, using a different method.

\subsection{Number fields and rings}

Let $F$ be a number field, denote by $\AAA$ its adele ring, by $\frak{o}$ the ring of integers, by $\frak{d}$ the different, $D_F$ the discriminant. Denote by $F_j$ $(1\leq j\leq r+s)$ the archimedean completions of $F$ (up to $\RR$-isomorphism), for $j\leq r$, $F_j\cong\RR$, for $j>r$, $F_j\cong\CC$. We introduce the character $\psi$ on $\lfact{F}{\AAA}$ as follows. Let $\psi:\AAA\rightarrow S^1$ be the unique continuous additive character, which is trivial on $F$; on $F_{\infty}$, it agrees with $\psi_{\infty}(x)= \exp(2\pi i(x_1+...+x_r+x_{r+1}+\overline{x_{r+1}}+...+x_{r+s}+\overline{x_{r+s}}))$; and on $F_{\frak{p}}$, it is trivial on $\frak{d}_{\frak{p}}^{-1}$ and nontrivial on $\frak{p}^{-1}\frak{d}_{\frak{p}}^{-1}$. For a fractional ideal $\frak{c}$, denote by $\norm{\frak{c}}$ its norm.

\subsection{Matrix groups}

Given a ring $R$, we define the following subgroups of $\GLtwo{2}{R}$:
\begin{equation*}
Z(R)=\left\{\begin{pmatrix} a & \ \cr \ & a\end{pmatrix}: a\in R^{\times}\right\},\ 
B(R)=\left\{\begin{pmatrix} a & b \cr \ & d\end{pmatrix}: a,d\in R^{\times}, b\in R\right\},\ 
N(R)=\left\{\begin{pmatrix} 1 & b \cr \ & 1\end{pmatrix}: b\in R\right\}.
\end{equation*}
Assume $0\neq\frak{n}_{\frak{p}},\frak{c}_{\frak{p}}\subseteq\frak{o}_{\frak{p}}$. Then let
\begin{equation*}
K_{\frak{p}}(\frak{n}_{\frak{p}},\frak{c}_{\frak{p}})=\left\{\begin{pmatrix} a & b \cr c & d\end{pmatrix}: a,d\in \frak{o}_{\frak{p}}, b\in(\frak{n}_\frak{p}\frak{d}_\frak{p})^{-1}, c\in \frak{n}_\frak{p}\frak{d}_\frak{p}\frak{c}_\frak{p}, ad-bc\in\frak{o}_\frak{p}^{\times}\right\},
\end{equation*}
moreover in the special case $\frak{n}_{\frak{p}}=\frak{o}_{\frak{p}}$, we simply write $K_{\frak{p}}(\frak{c}_{\frak{p}})$ instead of $K_{\frak{p}}(\frak{o}_{\frak{p}},\frak{c}_{\frak{p}})$. For an ideal $0\neq\frak{c}\subseteq\frak{o}$, let
\begin{equation*}
K(\frak{c})=\prod_{\frak{p}}K_{\frak{p}}(\frak{c}_{\frak{p}}),
\end{equation*}
and taking the archimedean places into account, let
\begin{equation*}
K=K(F_{\infty})\times K(\frak{o})\subseteq\GLtwo{2}{\AAA},
\end{equation*}
where
\begin{equation*}
K(F_{\infty})=(\mathrm{SO}_2(\RR))^r\times(\mathrm{SU}_2(\CC))^s.
\end{equation*}
Finally, for $0\neq\frak{n},\frak{c}\subseteq\frak{o}$, let
\begin{equation*}
\Gamma(\frak{n},\frak{c})=\left\{g_{\infty}\in\GLtwo{2}{F_{\infty}}: \exists g_{\fin}\in\prod_{\frak{p}}K_{\frak{p}}(\frak{n}_{\frak{p}},\frak{c}_{\frak{p}}) \mathrm{\ such\ that\ } g_{\infty}g_{\fin}\in\GLtwo{2}{F}\right\}.
\end{equation*}
We note that the choice of the subgroups $K$'s is not canonical (they can be conjugated arbitrarily), our normalization follows \cite{BlomerHarcos}.

\subsection{Measures}

First we remark that for the purpose of most applications, one does not need the exact normalization of our measures. However, for the sake of completeness, we give them, again mainly following \cite{BlomerHarcos}. On $F_{\infty}$, we use the Haar measure $|D_F|^{-1/2}dx_1\cdots dx_r|dx_{r+1}\wedge d\overline{x_{r+1}}|\cdots|dx_{r+s}\wedge d\overline{x_{r+s}}|$. On $F_{\frak{p}}$, we normalize the Haar measure such that $\frak{o}_{\frak{p}}$ has measure $1$. On $\AAA$ we use the Haar measure $dx$, the products of these measures, this induces a Haar probability measure on $\lfact{F}{\AAA}$ (see \cite[Chapter V, Proposition 7]{Weil}).

On $\RR^{\times}$, we use the measure $d_{\RR}^{\times}y=dy/|y|$, this gives rise to a measure on $\CC^{\times}$ as $d_{\CC}^{\times}y=d_{\RR}^{\times}|y|d\theta/2\pi$, where $\exp(i\theta)=y/|y|$. On $F_{\infty}^{\times}$, we use the product $d_{\infty}^{\times}y$ of these measures. On $F_{\frak{p}}^{\times}$ we normalize the Haar measure such that $\frak{o}_{\frak{p}}^{\times}$ has measure $1$. The product $d^{\times}y$ of these measures is a Haar measure on $\AAA^{\times}$, inducing some Haar measure on $\lfact{F^{\times}}{\AAA^{\times}}$.

On $K$ and its factors, we use the Haar probability measures. On $\lfact{Z(F_{\infty})}{\GLtwo{2}{F_{\infty}}}$, we use the Haar measure which satisfies
\begin{equation*}
\int_{\lfact{Z(F_{\infty})}{\GLtwo{2}{F_{\infty}}}}f(g)dg= \int_{(\RR^{\times})^r\times(\RR_+^{\times})^s} \int_{F_{\infty}} \int_{K(F_{\infty})}f\left(\begin{pmatrix}y & x \cr \ & 1 \end{pmatrix}k\right)dkdx\frac{d_{\infty}^{\times}y}{|y|},
\end{equation*}
where $|y|=\prod_{j=1}^r|y_j|\prod_{j=r+1}^{r+s}|y_j|^2$.

On $\GLtwo{2}{F_{\frak{p}}}$ we normalize the Haar measure such that $K(\frak{o}_{\frak{p}})$ has measure $1$. On $\lfact{Z(F_{\infty})}{\GLtwo{2}{\AAA}}$ we use the product of these measures.

As we noted above, the exact normalization is not really relevant. Because of this and for simplicity, in many calculations below, we will write $\const$ to denote a constant which is absolute in the sense that it depends only on the given number field $F$ and the normalization of measures.

\section{Special functions at archimedean places}

\subsection{Real places}

In the case $F_j\cong\RR$, we introduce the following functions. Assume
\begin{equation*}
\nu\in i\RR\cup\left(\ZZ+\frac{1}{2}\right) \cup\left(-\frac{1}{2},\frac{1}{2}\right), \quad \varepsilon\in\{0,1\}
\end{equation*}
are given numbers, $(\nu,\epsilon)$ later will be referred as the spectral parameter. Assume moreover, that $q\in 2\ZZ$ is a given even integer, later referred as the weight.

\begin{defi}\label{weightofrealvectors}
Assume $\phi:\GLtwo{2}{\RR}\rightarrow\CC$. We say it is of weight $q$, if
\begin{equation*}
\phi\left(g \begin{pmatrix} \cos\theta & \sin\theta \cr -\sin\theta & \cos\theta \end{pmatrix} \right)=e^{iq\theta}\phi(g)
\end{equation*}
for all $\theta\in\RR$.
\end{defi}

\begin{defi} Define the Whittaker function $\whit_{q,\nu}$ on $\RR^{\times}$ as
\begin{equation*}
\whit_{q,\nu}(y)=\frac{i^{\sign(y)\frac{q}{2}} W_{\sign(y)\frac{q}{2},\nu}(4\pi|y|)}{ (\Gamma(\frac{1}{2}-\nu+\sign(y)\frac{q}{2}) \Gamma(\frac{1}{2}+\nu+\sign(y)\frac{q}{2}))^{\frac{1}{2}}},
\end{equation*}
$W$ denoting the Whittaker function. We note that either $\Gamma(\frac{1}{2}-\nu+\sign(y)\frac{q}{2}) \Gamma(\frac{1}{2}+\nu+\sign(y)\frac{q}{2})\geq 0$ or it shows a pole.
\end{defi}
Compare this with \cite[pp.11-12]{BlomerHarcos}. We also record \cite[(25)]{BlomerHarcos}:
\begin{equation}\label{normofrealwhittaker}
\int_{\RR^{\times}}|\whit_{q,\nu}(y)|^2d_{\RR}^{\times}y=1.
\end{equation}

\subsection{Complex places}

In the case $F_j\cong\CC$, the available literature is much smaller, so we quote the details up to some extent in this and the following few subsections. We mainly follow the works of Bruggeman, Motohashi and Lokvenec-Guleska \cite{BruggemanMotohashi}, \cite{BruggemanMotohashi13}, \cite{Lokvenec}.

On the group $\mathrm{SU}_2(\CC)$, we normalize the Haar measure such that $\mathrm{SU}_2(\CC)$ has measure $1$. We need those irreducible representations of $\mathrm{SU}_2(\CC)$ that are trivial on the center (representations of $\mathrm{PSU}_2(\CC)$). These are parametrized by the nonnegative integers as follows. For all integer $l\geq 0$, there is a unique $2l+1$ dimensional irreducible representation and we now give the matrix coefficients $\Phi_{p,q}^l$ for $|p|,|q|\leq l$, where $p,q$ are integers. First observe
\begin{equation*}
\mathrm{SU}_2(\CC)=\left\{k[\alpha,\beta]= \begin{pmatrix}\alpha & \beta \cr -\overline{\beta} & \overline{\alpha} \end{pmatrix} : |\alpha|^2+|\beta|^2=1 \right\}.
\end{equation*}
Then the matrix coefficients are defined via the equation
\begin{equation}\label{su2matrixcoefficients}
\sum_{|p|\leq l}\Phi_{p,q}^l(k[\alpha,\beta])z^{l-p}= (\alpha z-\overline{\beta})^{l-q}(\beta z+\overline{\alpha})^{l+q}.
\end{equation}

Again, assume the spectral parameter $(\nu,p)$ is given such that
\begin{equation*}
(\nu\in i\RR, \quad p\in\ZZ) \quad \mathrm{OR} \quad \left(\nu\in\left(-\frac{1}{2},\frac{1}{2}\right)\setminus\{0\},\quad p=0\right),
\end{equation*}
and the weight $(l,q)$ is also given such that $l\geq |p|$ and $|q|\leq l$.

Recall the Iwasawa decomposition: any element $g\in\SLtwo{2}{\CC}$ can be uniquely written in the form $g=n(x)a(y)k[\alpha,\beta]$, where $k$ is defined above,
\begin{equation*}
n(x)= \begin{pmatrix} 1 & x \cr \ & 1\end{pmatrix}, \quad a(y)=\begin{pmatrix} \sqrt{y} & \ \cr \ & 1/\sqrt{y}\end{pmatrix},
\end{equation*}
$x\in\CC$, $y>0$ real.

First let
\begin{equation*}
\varphi(n(x)a(y)k[\alpha,\beta])=y^{1+\nu}\Phi_{p,q}^l(k[\alpha,\beta]).
\end{equation*}
When it is needed, we indicate the dependence on the weight and spectral data and write $\varphi_{l,q}(\nu,p)$.

\subsection{Jacquet integral}

For $\omega\in\CC$, and $f\in C^{\infty}(G)$ satisfying the growth condition
\begin{equation*}
f(n(x)a(y)k)=O(y^{1+\sigma})
\end{equation*}
with some $\sigma>0$, define the Jacquet integral
\begin{equation*}
\J_{\omega}f(g)=\int_{\CC} e^{-2\pi i(\omega x+\overline{\omega x})}f(\weyl n(x)g)dx
\end{equation*}
where $dx=d\Re xd\Im x$, $\weyl=k[0,1]$ stands for the Weyl element. For $0\neq\omega\in\CC$ and $f=\varphi$, we drop $\varphi$ from the notation and simply write $\J_{\omega}$ in place of $\J_{\omega}\varphi$. This can be computed (see \cite[Section 5]{BruggemanMotohashi} and \cite[Section 4.1]{Lokvenec}) to be
\begin{equation}\label{jacquettransform1}
\begin{split}
\J_{\omega}(n(x)a(y)k[\alpha,\beta])&=(-1)^{l-p} (2\pi)^{\nu} |\omega|^{\nu-1} e^{2\pi i(\omega x+\overline{\omega x})}\\ &\cdot\sum_{|m|\leq l} \left(\frac{i\omega}{|\omega|}\right)^{-p-m} w^l_m(\nu,p;|\omega|y)\Phi_{m,q}^l(k[\alpha,\beta]),
\end{split}
\end{equation}
where
\begin{equation}\label{jacquettransform2}
w^l_m(\nu,p;y)=\sum_{j=0}^{l-\frac{1}{2}(|m+p|+|m-p|)}(-1)^j \xi^l_p(m,j)\frac{(2\pi y)^{l+1-j}}{\Gamma(l+1+\nu-j)}K_{\nu+l-|m+p|-j}(4\pi y),
\end{equation}
$K$ denoting the $K$-Bessel function, and
\begin{equation}\label{jacquettransform3}
\xi^l_p(m,j)=\frac{j!(2l-j)!}{(l-p)!(l+p)!} {{l-\frac{1}{2}(|m+p|+|m-p|)}\choose{j}} {{l-\frac{1}{2}(|m+p|-|m-p|)}\choose{j}}.
\end{equation}
Note that a priori we need $\Re\nu>0$, but we can remove this condition by analytic continuation.

\subsection{Goodman-Wallach operator}

Another operator we need, is the Goodman-Wallach operator, which we specialize again to $\varphi$ and obtain (see \cite[Section 6]{BruggemanMotohashi} and \cite[Section 4.2]{Lokvenec})
\begin{equation*}
\begin{split}
\M_{\omega}(n(x)a(y)k[\alpha,\beta])&= (2\pi|\omega|)^{-\nu-1} e^{2\pi i(\omega x+\overline{\omega x})}\\ &\cdot\sum_{|m|\leq l} \left(\frac{-i\omega}{|\omega|}\right)^{p-m} \mu^l_m(\nu,p;|\omega|y)\Phi_{m,q}^l(k[\alpha,\beta]),
\end{split}
\end{equation*}
where
\begin{equation*}
\mu^l_m(\nu,p;y)=\sum_{j=0}^{l-\frac{1}{2}(|m+p|+|m-p|)} \xi^l_p(m,j)\frac{(2\pi y)^{l+1-j}}{\Gamma(l+1+\nu-j)}I_{\nu+l-|m+p|-j}(4\pi y),
\end{equation*}
$I$ denoting the $I$-Bessel function. Again, occasionally, we may write $\M_{\omega}\varphi$ or even $\M_{\omega}\varphi_{l,q}(\nu,p)$, when there is any danger of confusion.

We cite \cite[Lemma 6.2]{BruggemanMotohashi}, \cite[Lemma 4.2.2]{Lokvenec} for the relations matching these operators. Here, we have to indicate the dependence on $\nu,p$.
\begin{lemm}\label{jacquettransformofgoodmanwallach}
Let $\omega_2\neq 0$, $\Re\nu>0$. Then
\begin{equation*}
\J_0\M_{\omega_2}=\frac{\sin\pi(\nu-p)}{\nu^2-p^2} \frac{\Gamma(l+1-\nu)}{\Gamma(l+1+\nu)}\varphi(-\nu,-p).
\end{equation*}
and for $\omega_1\neq 0$,
\begin{equation*}
\J_{\omega_1}\M_{\omega_2} =\mathcal{J}^*_{\nu,p}(4\pi\sqrt{\omega_1\omega_2})\J_{\omega_1},
\end{equation*}
with
\begin{equation*}
\mathcal{J}^*_{\nu,p}(z)=J^*_{\nu-p}(z)J^*_{\nu+p}(\overline{z}),
\end{equation*}
where $J^*_{\nu}$ is the even entire function of $z$ which is equal to $J_{\nu}(z)(z/2)^{-\nu}$ for $z>0$ and $J$ stands for the $J$-Bessel function. \qed
\end{lemm}

\subsection{The Whittaker function}

We would like to define the Whittaker function analogously to the real case, following again \cite[Sections 3-5]{BruggemanMotohashi}, \cite[Chapters 2-4]{Lokvenec}. The real Lie algebra $\frak{sl}_2(\CC)$ of $\SLtwo{2}{\CC}$ is generated by the elements
\begin{equation*}
\mathbf{H}_1=\frac{1}{2}\begin{pmatrix}1 & \ \\ \ & -1\end{pmatrix},
\mathbf{V}_1=\frac{1}{2}\begin{pmatrix}\ & 1 \\ 1 & \ \end{pmatrix},
\mathbf{W}_1=\frac{1}{2}\begin{pmatrix}\ & 1 \\ -1 & \ \end{pmatrix},
\end{equation*}
\begin{equation*}
\mathbf{H}_2=\frac{1}{2}\begin{pmatrix}i & \ \\ \ & -i\end{pmatrix},
\mathbf{V}_2=\frac{1}{2}\begin{pmatrix}\ & i \\ -i & \ \end{pmatrix},
\mathbf{W}_2=\frac{1}{2}\begin{pmatrix}\ & i \\ i & \ \end{pmatrix}.
\end{equation*}
Let $\frak{g}=\frak{sl}_2(\CC)\otimes_{\RR}\CC$ be the complexification of $\mathrm{sl}_2(\CC)$. The center of the universal enveloping algebra $U(\frak{g})$ is generated by two Casimir elements:
\begin{equation*}
\Omega_{\pm}=\frac{1}{8}\left(\left(\mathbf{H}_1\mp \mathbf{H}_2\right)^2+\left(\mathbf{V}_1\mp \mathbf{W}_2\right)^2+\left(\mathbf{W}_1\mp \mathbf{V}_2\right)^2\right).
\end{equation*}
By a Whittaker function on $\SLtwo{2}{\CC}$, we mean a function which is an eigenfunction of $\Omega_{\pm}$, has exponential decay as $y\rightarrow\infty$ and on which the unipotent subgroup $\{n(x):x\in\CC\}$ acts through a character of $\CC$.

For any fixed $0\neq\omega\in\CC^{\times}$, the function $\J_{\omega}$ satisfies these conditions. However, we shall fix one of them and suitably normalize it. The following Lemma is analogous to \cite[Theorem 2]{BruggemanMotohashi13}, the difference arises from the normalization and the fact that we work it out for the complementary series as well (i.e. we do not require $\Re\nu=0$).

\begin{lemm}\label{normofjacquettransform}
\begin{equation*}
\int_0^{\infty}|\J_1(a(y))|^2\frac{dy}{y}= \frac{(2\pi)^{2\Re\nu}}{8(2l+1)}{{2l}\choose{l-q}}^{-1}{{2l}\choose{l-p}} \left|\frac{\Gamma(l+1-\nu)}{\Gamma(l+1+\nu)}\right|.
\end{equation*}
\end{lemm}
\begin{proof}
First observe that by (\ref{su2matrixcoefficients}), $\Phi_{m,q}^l(k[1,0])=0$, if $m\neq q$, and $\Phi_{q,q}^l(k[1,0])=1$. Hence in \cite[(4.8)]{Lokvenec}, we may write
\begin{equation*}
\J_1(a(y))=v^l_q(y,1),
\end{equation*}
where
\begin{equation*}
v^l_q(y,1)=y^{1-\nu}\int_{\CC} \frac{e^{-2\pi iy(z+\overline{z})}}{(1+|z|^2)^{1+\nu}} \Phi_{p,q}^l \left(k\left[\frac{\overline{z}}{\sqrt{1+|z|^2}}, \frac{-1}{\sqrt{1+|z|^2}}\right]\right)dz.
\end{equation*}
So what is left is to compute $\int_0^{\infty}|v^l_q(y,1)|^2dy/y$.

First let $q=l$. Then by (\ref{jacquettransform1}), (\ref{jacquettransform2}), (\ref{jacquettransform3}), this integral is
\begin{equation}\label{belso1}
 (2\pi)^{\nu+\overline{\nu}}\frac{1}{\Gamma(l+1+\nu)\Gamma(l+1+\overline{\nu})} \left(\frac{(2l)!}{(l-p)!(l+p)!}\right)^2\int_0^{\infty}(2\pi y)^{2l+2} |K_{\nu-p}(4\pi y)|^2\frac{dy}{y}.
\end{equation}
To compute the inner integral, we use \cite[6.576(4)]{GR}:
\begin{equation*}
\begin{split}
&\int_0^{\infty}r^{2l+1}|K_{\nu}(r)|^2dr=\\ &\frac{2^{2l-1}}{(2l+1)!} \Gamma\left(l+1+\frac{\nu}{2}-\frac{\overline{\nu}}{2}\right) \Gamma\left(l+1-\frac{\nu}{2}+\frac{\overline{\nu}}{2}\right)
\Gamma\left(l+1+\frac{\nu}{2}+\frac{\overline{\nu}}{2}\right)
\Gamma\left(l+1-\frac{\nu}{2}-\frac{\overline{\nu}}{2}\right),
\end{split} 
\end{equation*}
the conditions are all satisfied by noting $|\Re\nu|<1/2$.

First assume we are in the principal series ($\Re\nu=0$). Then (\ref{belso1}) equals
\begin{equation*}
\begin{split}
\frac{1}{2^{2l+2}} \frac{1}{\Gamma(l+1+\nu)\Gamma(l+1-\nu)} \left(\frac{(2l)!}{(l-p)!(l+p)!}\right)^2 \frac{2^{2l-1}}{(2l+1)!} \Gamma(l+1+\nu) &\Gamma(l+1-\nu)(l+p)!(l-p)!=\\ &\frac{1}{8(2l+1)}{{2l}\choose{l-p}}.
\end{split}
\end{equation*}
Now assume we are in the complementary series ($\Im\nu=0$, $\nu\neq 0$). Then $p=0$ and for (\ref{belso1}), we obtain
\begin{equation*}
\begin{split}
\frac{(2\pi)^{2\nu}}{2^{2l+2}} \frac{1}{\Gamma(l+1+\nu)\Gamma(l+1+\nu)} \left(\frac{(2l)!}{(l!)^2}\right)^2 \frac{2^{2l-1}}{(2l+1)!} \Gamma(l+1+\nu) &\Gamma(l+1-\nu)(l!)^2=\\ &\frac{(2\pi)^{2\nu}}{8(2l+1)}{{2l}\choose{l}} \frac{\Gamma(l+1-\nu)}{\Gamma(l+1+\nu)}.
\end{split}
\end{equation*}
For a general $|q|\leq l$, the identity \cite[p.89]{BruggemanMotohashi13}
\begin{equation*}
\int_0^{\infty}|v^l_q(y)|^2\frac{dy}{y}= {{2l}\choose{l-q}}^{-1}\int_0^{\infty}|v^l_l(y)|^2\frac{dy}{y}
\end{equation*}
completes the proof.
\end{proof}
Now we are ready to introduce the appropriate normalization.
\begin{defi}
For $y>0$, the Whittaker function is defined as
\begin{equation*}
\whit_{l,q,\nu,p}(y)=\J_1(a(y)) \frac{\sqrt{8(2l+1)}}{(2\pi)^{\Re\nu}}{{2l}\choose{l-q}}^{\frac{1}{2}} {{2l}\choose{l-p}}^{-\frac{1}{2}} \sqrt{\left|\frac{\Gamma(l+1+\nu)}{\Gamma(l+1-\nu)}\right|},
\end{equation*}
and for an arbitrary $y\in\CC^{\times}$, as
\begin{equation*}
\whit_{l,q,\nu,p}(y)=\whit_{l,q,\nu,p}(|y|)\left(\frac{y}{|y|}\right)^{-q}.
\end{equation*}
\end{defi}
Now Lemma \ref{normofjacquettransform} gives the complex analog of (\ref{normofrealwhittaker}):
\begin{equation}\label{normofcomplexwhittaker}
\int_{\CC^{\times}}|\whit_{l,q,\nu,p}(y)|^2d_{\CC}^{\times}y=1.
\end{equation}

Finally, we define the complex weight, the counterpart of the notion introduced in Definition \ref{weightofrealvectors}.
\begin{defi}\label{weightofcomplexvectors}
Assume $\phi:\GLtwo{2}{\CC}\rightarrow\CC$. We say it is of weight $(l,q)$, if it is smooth, and for the derived action of the Lie algebra $\frak{g}$,
\begin{equation*}
\mathbf{H}_2\phi=-iq\phi, \quad \Omega_{\frak{k}}\phi=-\frac{1}{2}(l^2+l)\phi,
\end{equation*}
where
\begin{equation*}
\Omega_{\frak{k}}=-\frac{1}{2}(\mathbf{H}_2^2+\mathbf{W}_1^2+\mathbf{W}_2^2).
\end{equation*}
\end{defi}

\section{Statement of the Kuznetsov formula}

We still need some further notations before stating the theorem, which we present in the following subsections.

\subsection{Fourier-Whittaker expansion}

Let $V_{\pi}$ be an irreducible automorphic representation generated by a cusp form or an Eisenstein series in the Hilbert space $L^2(\lfact{Z(\AAA)\GLtwo{2}{F}}{\GLtwo{2}{\AAA}})$.

Assume $\phi\in V_{\pi}$ is a smooth vector of pure weight, that is, at each archimedean quasifactor it is of pure weight, $q$ or $(l,q)$ (depending on $F_j$).

Then $\phi$ has a Fourier-Whittaker expansion, that is, for $y\in\AAA^{\times}$, $x\in\AAA$,
\begin{equation}\label{fourierwhittakerexpansion}
\phi\left( \begin{pmatrix} y & x \cr \ & 1 \end{pmatrix} \right)=\rho_{\phi,0}(y)+ \sum_{r\in F^{\times}}\rho_{\phi}(ry_{\fin}) (\sign(ry_{\infty}))^{\varepsilon}\whit(ry_{\infty})\psi(rx),
\end{equation}
where $\whit=\prod_{j=1}^{r+s}\whit_j$, where $\whit_j=\whit_{q_j,\nu_j}$ at real, $\whit_j=\whit_{l_j,q_j,\nu_j,p_j}$ at complex places. The spectral parameter ($\nu_j$ or $(\nu_j,p_j)$) is determined by $\pi$, the weight ($q_j$ or $(l_j,q_j$) is the weight of $\phi$, and $\varepsilon$ is a sign character coming from the $\varepsilon$-part of the real spectral parameter (note that it is not well-defined for discrete series representations).

Here, $\rho_{\phi}(ry_{\fin})$ depends only (apart from $\phi$) on the fractional ideal generated by $ry_{\fin}$ and it is zero if this ideal is nonintegral. We also note that if $V_{\pi}$ is cuspidal, then $\rho_{\phi,0}(y)=0$.

We normalize these coefficients by introducing $\lambda_{\phi}(\frak{m})$:
\begin{equation*}
\lambda_{\phi}(\frak{m})=\rho_{\phi}(\frak{m}) \sqrt{\norm{\frak{m}}}.
\end{equation*}

\subsection{Kloosterman sums}

We qoute the definition of Kloosterman sums from \cite[Definition 2]{Venkatesh}.
\begin{defi} Let $\frak{a}_1,\frak{a}_2$ be fractional ideals of $F$, and $\frak{c}$ be any ideal such that $\frak{c}^2\sim\frak{a}_1\frak{a}_2$ (i.e. they are in the same ideal class). Let then $c\in \frak{c}^{-1}$, $\alpha_1\in\frak{a}_1^{-1}\frak{d}^{-1}$, $\alpha_2\in\frak{a}_1\frak{d}^{-1}\frak{c}^{-2}$. We define the Kloosterman sum as
\begin{equation*}
KS(\alpha_1,\frak{a}_1; \alpha_2,\frak{a}_2; c,\frak{c})= \sum_{x\in (\rfact{\frak{a}_1\frak{c}^{-1}}{\frak{a}_1c})^{\times}} \psi_{\infty}\left(\frac{\alpha_1 x+\alpha_2 x^{-1}}{c}\right),
\end{equation*}
where the summation runs through the $x$'s which generate $\rfact{\frak{a}_1\frak{c}^{-1}}{\frak{a}_1c}$ as an $\frak{o}$-module, and $x^{-1}$ is the unique element in $(\rfact{\frak{a}_1^{-1}\frak{c}}{\frak{a}_1^{-1}c\frak{c}^2})^{\times}$ such that $xx^{-1}\in 1+c\frak{c}$.
\end{defi}

\subsection{Archimedean Bessel transforms and measures}\label{archimedeanbesseltransformsandmeasures}

In the Kuznetsov formula, on the so-called geometric side, the weight functions are Bessel transforms. Assume $f(\nu,p)$ is a function of the form
\begin{equation*}
f(\nu,p)=\prod_{j=1}^r f_j(\nu_j) \prod_{j=r+1}^{r+s} f_j(\nu_j,p_j),
\end{equation*}
where $f_j$'s are functions on the possible spectral parameter values: $\nu_j$ in the real case, $(\nu_j,p_j)$ in the complex case. Then let
\begin{equation*}
\mathcal{B}f_{(\nu,p)}(z)=\prod_{j=1}^{r}(\mathcal{B}_jf_j)_{\nu_j}(z) \prod_{j=r+1}^{r+s}(\mathcal{B}_jf_j)_{(\nu_j,p_j)}(z),
\end{equation*}
where $\mathcal{B}_jf_j$ is defined as follows. For $F_j\cong\RR$,
\begin{equation*}
\begin{split}
&(\mathcal{B}_jf_j)_{\nu_j}(z)=f_j(\nu_j)\cdot(\mathcal{B}_j)_{\nu_j}(z),\\ &(\mathcal{B}_j)_{\nu_j}(z)=\frac{2\pi}{\sin \pi\nu_j}(J_{-2\nu_j}(|z|)-J_{2\nu_j}(|z|)),
\end{split}
\end{equation*}
$J$ standing for the $J$-Bessel function.
For $F_j\cong\CC$,
\begin{equation*}
\begin{split}
&(\mathcal{B}_jf_j)_{(\nu_j,p_j)}(z)=f_j(\nu_j,p_j) \cdot(\mathcal{B}_j)_{(\nu_j,p_j)}(z),\\ &(\mathcal{B}_j)_{(\nu_j,p_j)}(z)= \frac{|z/2|^{-2\nu_j}(iz/|z|)^{2p_j}\mathcal{J}^*_{-\nu_j,-p_j}(z) -|z/2|^{2\nu_j}(iz/|z|)^{-2p_j}\mathcal{J}^*_{\nu_j,p_j}(z)}{\sin \pi(\nu_j-p_j)},
\end{split}
\end{equation*}
$\mathcal{J}^*$ is defined in Lemma \ref{jacquettransformofgoodmanwallach}.

Introduce moreover the measure $d\mu$ on the space of spectral parameters as follows. Again, we give it locally: $d\mu=\prod_j d\mu_j$. For $F_j\cong\RR$,
\begin{equation}\label{definitionofmureal}
\int f(\nu_j) d\mu_j(\nu_j)= \int_0^{i\infty}f(\nu_j)(-4\pi\nu_j)\tan\pi\nu_j \frac{d\nu_j}{2\pi i}+\sum_{2|2\nu_j+1,1<2\nu_j+1}f(\nu_j).
\end{equation}
For $F_j\cong\CC$,
\begin{equation}\label{definitionofmucomplex}
\int f(\nu_j,p_j) d\mu_j(\nu_j,p_j)= \sum_{p_j}\int_{(0)} f(\nu_j,p_j) (p_j^2-\nu_j^2)d\nu_j.
\end{equation}

\subsection{The Kuznetsov formula}\label{The Kuznetsov formula}

Let $h=\prod_jh_j$, where $h_j$'s are defined as follows. Let $a_j>1$ be given. Then at real places
\begin{equation*}
h_j(\nu_j)=\left\{\begin{array}{ll}e^{(\nu_j^2-\frac{1}{4})/a_j} & \mathrm{if}\ |\Re\nu_j|<\frac{2}{3}, \cr 1 & \mathrm{if}\ \nu_j\in\frac{1}{2}+\ZZ, \frac{3}{2}\leq|\nu_j|\leq a_j, \cr 0 & \mathrm{otherwise}.\end{array}\right.
\end{equation*}
While at complex places
\begin{equation*}
h_j(\nu_j,p_j)=\left\{\begin{array}{ll}e^{(\nu_j^2+p_j^2-1)/a_j} & \mathrm{if}\ |\Re\nu_j|<\frac{2}{3}, p_j\in\ZZ, |p_j|\leq a_j, \cr 0 & \mathrm{otherwise}.\end{array}\right.
\end{equation*}

Choose moreover an orthonormal basis $\mathbf{B}$ in the cuspidal spectrum consisting of pure weight forms $\mathbf{f}$. For an ideal $\frak{c}$, let us denote by $\mathbf{B}(\frak{c})$ the set of the basis elements $\mathbf{f}$ which are right $K(\frak{c})$-invariant: $\mathbf{f}(gk)=\mathbf{f}(g)$, if $k\in K(\frak{c})$. The spectral side of the formula also involves some contribution of the continuous spectrum, which is analogous to the cuspidal part. We will not spell this out explicitly: following \cite{Venkatesh}, denote it simply by $CSC$.

Fix some fractional ideals $\frak{a}^{-1},\frak{a}'^{-1}$, some nonzero elements $\alpha\in\frak{ad}^{-1},\alpha\in\frak{a}'\frak{d}^{-1}$. Let $C$ be a fixed set of narrow ideal class representatives $\frak{m}$, for which $\frak{m}^2\frak{aa}'^{-1}$ is a principal ideal generated by a totally positive element $\gamma_{\frak{m}}$, fixed once for all.

\begin{theo}\label{kuznetsov}
The sum formula holds for the weight function $h$, that is,
\begin{equation}\label{kuznetsovformula}
\begin{split}
&\left[K(\frak{o}):K(\frak{c})\right]^{-1} \sum_{\mathbf{f}\in\mathbf{B}(\frak{c})} h(\nu,p) \lambda_{\mathbf{f}}(\alpha\frak{a}^{-1}) \overline{\lambda_{\mathbf{f}}(\alpha'\frak{a}'^{-1})}+CSC=\\
&\const\Delta(\alpha\frak{a}^{-1},\alpha'\frak{a}'^{-1}) \int h(\nu,p) d\mu+\\ &\const \sum_{\frak{m}\in C}\sum_{c\in \frak{amc}}\sum_{\epsilon\in\rfact{\frak{o}_+^{\times}}{\frak{o}^{2\times}}} \frac{KS(\epsilon\alpha,\frak{a}^{-1}\frak{d}^{-1};\alpha'\gamma_{\frak{m}}, \frak{a}'^{-1}\frak{d}^{-1};c,\frak{a}^{-1}\frak{m}^{-1}\frak{d}^{-1}) }{\norm{c\frak{a}^{-1}\frak{m}^{-1}}}
\int\mathcal{B}h_{(\nu,p)}\left(\frac{|\alpha\alpha'\gamma_\frak{m} \epsilon|^{\frac{1}{2}}}{c}\right) d\mu,
\end{split}
\end{equation}
where $\Delta(\alpha\frak{a}^{-1},\alpha'\frak{a}'^{-1})$ is $1$ if $\alpha\frak{a}^{-1}=\alpha'\frak{a}'^{-1}$, and $0$ otherwise; $\frak{o}_+^{\times}$ stands for the group of totally positive units; for the integrals with respect to $d\mu$, recall (\ref{definitionofmureal}) and (\ref{definitionofmucomplex}). The constants denoted by $\const$ are nonzero and depend only on the field $F$ and the normalization of measures.
\end{theo}

\section{Poincar\'e series}\label{Poincare series}

\subsection{Transition between adelic and classical forms}\label{Transition between adelic and classical forms}

First we match the adelic automorphic forms with classical ones. We borrow the transition from \cite[Section 2.12]{BlomerHarcos}.

For any nonzero ideal $\frak{n}\subseteq\frak{o}$, let $\eta\in\AAA_{\fin}^{\times}$ be a finite idele representing $\frak{n}$ and now observe that
\begin{equation*}
\Gamma(\frak{n},\frak{c})Z(F_{\infty})g\mapsto \GLtwo{2}{F}Z(F_{\infty})g\begin{pmatrix}\eta^{-1} & \ \cr & 1\end{pmatrix}K(\frak{c}),\quad g\in\GLtwo{2}{F_{\infty}}
\end{equation*}
gives an embedding
\begin{equation*}
\lfact{\Gamma(\frak{n},\frak{c})Z(F_{\infty})}{\GLtwo{2}{F_{\infty}}} \hookrightarrow \lrfact{\GLtwo{2}{F} Z(F_{\infty})}{\GLtwo{2}{\AAA}}{K(\frak{c})}.
\end{equation*}
Using now strong approximation \cite[Theorem 3.3.1]{Bump}, and taking ideal class representatives $\frak{n}_1,\ldots,\frak{n}_h$, we obtain a decomposition
\begin{equation}\label{decompositiontransition}
\lrfact{\GLtwo{2}{F} Z(F_{\infty})}{\GLtwo{2}{\AAA}}{K(\frak{c})}\cong \coprod_{j=1}^h \lfact{\Gamma(\frak{n}_j,\frak{c})Z(F_{\infty})}{\GLtwo{2}{F_{\infty}}}.
\end{equation}
\begin{lemm}\label{measuretransition}
Using the measures induced by those we defined earlier, for any Borel set $U$ in the decomposition (\ref{decompositiontransition}),
\begin{equation*}
\mathrm{measure_{LHS}}(U)= [K(\frak{o}):K(\frak{c})]^{-1}\mathrm{measure_{RHS}}(U).
\end{equation*}
\end{lemm}
\begin{proof}
See \cite[p.34]{BlomerHarcos}.
\end{proof}
Now we turn to the automorphic forms. Let
\begin{equation*}
\begin{split}
\FS&= \left\{f:\GLtwo{2}{\AAA} \rightarrow\CC:\int_{\lfact{Z(\AAA)\GLtwo{2}{F}}{\GLtwo{2}{\AAA}}} |f|^2<\infty,\right.\\ &\left.f\left(\gamma\begin{pmatrix}z & \ \cr \ & z \end{pmatrix}gk\right)=f(g),\mathrm{\ if\ }\gamma\in\GLtwo{2}{F},z\in F_{\infty}^{\times},k\in K(\frak{c})\right\}.
\end{split}
\end{equation*}
This is a larger space of automorphic functions than one usually deals with: this is the $L^2$ space of the LHS of (\ref{decompositiontransition}). Then to any $\phi\in \FS$, (\ref{decompositiontransition}) associates automorphic functions, which we denote by
\begin{equation*}
\phi^{\frak{n}}(g_{\infty})=\phi\left(\begin{pmatrix}\eta^{-1} & \ \cr \ & 1\end{pmatrix}g_{\infty}\right), \quad g_{\infty}\in\GLtwo{2}{F_{\infty}}.
\end{equation*}

\subsection{Definition of Poincar\'e series}

Fix some nonzero ideals $\frak{a},\frak{b}$. We define the following characters on $N(F_{\infty})$. For $x\in F_{\infty}$, let $\psi_1(x)=\psi_{\infty}(\alpha x)$ and $\psi_2(x)=\psi_{\infty}(\alpha' x)$, where $\alpha\in\frak{ad}^{-1}$, $\alpha'\in\frak{ab}^2\frak{d}^{-1}$ are nonzero elements with the property that $\alpha/\alpha'$ it totally positive (that is, positive at all real places). They give rise naturally to characters of $N(F_{\infty})$ which are trivial on $\Gamma(\frak{a},\frak{c})\cap N(F_{\infty})$, $\Gamma(\frak{ab}^2,\frak{c})\cap N(F_{\infty})$, respectively.

The building blocks of the Poincar\'e series are functions $f_1,f_2$ on $\lfact{Z(F_{\infty})}{\GLtwo{2}{F_{\infty}}}$ with the prescribed left action of $N(F_{\infty})$:
\begin{equation}\label{naction}
f_1\left(\begin{pmatrix} 1 & x \cr \ & 1\end{pmatrix}\begin{pmatrix} z & \ \cr \ & z\end{pmatrix}g\right)=\psi_1(x)f_1(g),\quad
f_2\left(\begin{pmatrix} 1 & x \cr \ & 1\end{pmatrix}\begin{pmatrix} z & \ \cr \ & z\end{pmatrix}g\right)=\psi_2(x)f_2(g).
\end{equation}

Then the Poincar\'e series are defined on the LHS of (\ref{decompositiontransition}) as
\begin{equation}\label{poincareseries1}
P_1^{\frak{a}}(g)= \sum_{\gamma\in \lfact{Z_{\Gamma}\Gamma_N(\frak{a},\frak{c})}{\Gamma(\frak{a},\frak{c})}} f_1(\gamma g),\quad P_2^{\frak{a}\frak{b}^2}(g)= \sum_{\gamma\in \lfact{Z_{\Gamma}\Gamma_N(\frak{ab}^2,\frak{c})}{\Gamma(\frak{ab}^2,\frak{c})}} f_2(\gamma g),
\end{equation}
with $Z_{\Gamma}$ standing for the center, $\Gamma_N$ for the upper triangular unipotent subgroup (the intersection with $N(F_{\infty})$) of the corresponding group $\Gamma$; this defines both $P_1$ and $P_2$ only on a single component (in the decomposition (\ref{decompositiontransition})), on other components, let them be zero.

Of course, there might be convergence problems. If we can define the building blocks such that the resulting Poincar\'e series are absolutely convergent, then this definition is valid. Unfortunately, when all the archimedean places are complex, we are not able to guarantee the absolute convergence, so in this case, we have to clarify, what we mean by the sums in (\ref{poincareseries1}). We will return to this problem in the Appendix. Until then, we always indicate how this 'case $r=0$' can be treated.

Our building blocks will be pure tensors, $f_i(x)=\prod_{j=1}^{r+s}f_{i,j}(x)$ for $i=1,2$. In the next subsections, we give the local definitions.

\subsection{Building blocks at real places}\label{Building blocks at real places}

In the construction of a real factor of our building blocks, we mainly follow \cite[Section 3.2]{BruggemanMiatello}, where the authors work with the group $\SLtwo{2}{\RR}$.

For a given $\sigma\in(\frac{1}{2},1)$ and even integer $u$, we denote by $\mathcal{T}_{u,\sigma}$ the linear space of functions $\eta$ defined on the set
\begin{equation*}
\{\nu\in\CC:|\Re\nu|\leq\sigma\}\cup\left\{\frac{1}{2},\frac{3}{2},\ldots\right\}
\end{equation*}
and satisfying the conditions
\begin{enumerate}
\item $\eta$ is holomorphic and even on a neighborhood of the strip $|\Re\nu|\leq\sigma$,

\item $\eta(\nu)\ll e^{-\frac{\pi}{2}|\Im\nu|}(1+|\Im\nu|)^{-A}$ for each $A>0$,

\item $\eta\left(\frac{b-1}{2}\right)=0$, if $b$ is an even integer such that $b>u$.
\end{enumerate}

For $\eta\in\mathcal{T}_{u,\sigma}$ define the following function on the set $y>0$ (see \cite[(3.10)]{BruggemanMiatello})
\begin{equation*}
\begin{split}
(\widetilde{\mathcal{L}}_u\eta)(y)&= \frac{1}{4\pi i}\int_{(0)}\eta(\nu)W_{u/2,\nu}(y) \left| \frac{\Gamma\left(\frac{1}{2}+\nu-\frac{u}{2}\right)}{\Gamma(2\nu)}\right|^2d\nu\\
&+\sum_{2|b,1<b\leq u}\eta\left(\frac{b-1}{2}\right)W_{u/2,(b-1)/2}(y) \frac{b-1}{\left(\frac{u-b}{2}\right)!\left(\frac{u+b-2}{2}\right)!}.
\end{split}
\end{equation*}

Now for some fixed $q\in\ZZ,\alpha\in\RR\setminus\{0\}$ and $\eta\in\mathcal{T}_{u,\sigma}$, we define the following function using the Iwasawa decomposition. First, if $\det g>0$, then let
\begin{equation*}
(\widetilde{\mathcal{L}}_q^{\alpha}\eta)(g)=(\widetilde{\mathcal{L}}_q^{\alpha}\eta)\left(\begin{pmatrix}1 & x \cr \ & 1\end{pmatrix} \begin{pmatrix}y & \ \cr \ & 1\end{pmatrix}\begin{pmatrix}\cos\theta & \sin\theta \cr -\sin\theta & \cos\theta\end{pmatrix}\right)=e^{2\pi i\alpha x}(\widetilde{\mathcal{L}}_{q\sign(\alpha)}\eta)(4\pi|\alpha|y)e^{iq\theta}.
\end{equation*}
If $\det g<0$, then $g=g'\begin{pmatrix}-1 & \ \cr \ & 1\end{pmatrix}$ with $\det g'>0$ and in this case we simply prescribe $(\widetilde{\mathcal{L}}_q^{\alpha}\eta)(g) =(\widetilde{\mathcal{L}}_q^{-\alpha}\eta)(g')$.

Now let $\eta,\theta\in\mathcal{T}_{u,\sigma}$. For real factors of $f_1,f_2$, choose $\widetilde{\mathcal{L}}_q^{\alpha}\eta, \widetilde{\mathcal{L}}_q^{\alpha'}\theta$, respectively. Of course, we may use different $\eta$'s and $\theta$'s at different real places.

Before turning to complex places, note that the functions we defined transform like weight $q$ functions on the positive domain $\det g>0$, and like weight $-q$ functions on the negative domain $\det g<0$. 

\subsection{Building blocks at complex places}\label{Building blocks at complex places}

In the construction of a complex factor of our bulding blocks, we follow \cite[Section 7]{BruggemanMotohashi} and \cite[Section 9.1]{Lokvenec}.

Let $l>0$ be an integer, $|q|\leq l$. Following \cite[Theorem 7.1]{BruggemanMotohashi} and
\cite[Definition 9.1.3]{Lokvenec}, for a given $\sigma\in(1,\frac{3}{2})$, we denote by $\mathcal{T}^l_{\sigma}$ the linear space of functions $\eta$ defined
on the set
\begin{equation*}
\{(\nu,p)\in\CC\times\ZZ:|\Re\nu|\leq\sigma,|p|\leq l\}
\end{equation*}
and satisfying the following conditions
\begin{enumerate}
\item $\eta$ is holomorphic on a neighborhood of the strip $|\Re\nu|\leq\sigma$,

\item $\eta(\nu,p)\ll e^{-\frac{\pi}{2}|\Im\nu|}(1+|\Im\nu|)^{-A}$ for any $A>0$,

\item $\eta(\nu,p)=\eta(-\nu,-p)$.
\end{enumerate}

Now for the given $l,q$, some fixed $\alpha$ and $\eta\in\mathcal{T}^l_{\sigma}$, let
\begin{equation*}
\begin{split}
(\widetilde{\mathcal{L}}_{l,q}^{\alpha}\eta)(g)&=\frac{|\alpha|}{2\pi^3 i}\sum_{|p|\leq l}\frac{(i\alpha/|\alpha|)^p}{(2l+1)^{-1/2} {{2l}\choose{l-p}}^{1/2}{{2l}\choose{l-q}}^{-1/2}}\cdot\\ &\int_{(0)}\eta(\nu,p)(2\pi|\alpha|)^{-\nu}\Gamma(l+1+\nu) \J_{\alpha}(g)\nu^{\epsilon(p)}\sin\pi(\nu-p)d\nu
\end{split}
\end{equation*}
with $\epsilon(0)=1$, $\epsilon(p)=-1$ for $p\in\ZZ\setminus\{0\}$. Note that this function differs from the function appearing in \cite[Theorem 9.1.4]{Lokvenec} by the factor $|\alpha|$.

Now let $\eta,\theta\in\mathcal{T}^l_{\sigma}$. For a complex factor of $f_1$, choose $\widetilde{\mathcal{L}}_{l,q}^{\alpha}\eta$ and for $f_2$, choose $\widetilde{\mathcal{L}}_{l,q}^{\alpha'}\theta$. Of course, we may use different $\eta$'s and $\theta$'s at different complex places.

\section{Scalar product of Poincar\'e series}

\subsection{Geometric description}\label{Geometric description}

Let $\pi_{\frak{b}}$ be a finite idele representing $\frak{b}$. With the abbreviation
\begin{equation*}
\pi_{\frak{b}}^{-1}P_2= \begin{pmatrix} \pi_{\frak{b}}^{-1} & \cr & \pi_{\frak{b}}^{-1} \end{pmatrix}P_2,
\end{equation*}
consider the inner product $\langle \pi_{\frak{b}}^{-1}P_2,P_1\rangle$. The Poincar\'e series $P_1,P_2$ are defined in the space $\FS$, and the inner product is also understood there.
In what follows, we shall expand this both geometrically and spectrally, and the equation of these expressions will give rise to the Kuznetsov formula.

First we note the following consequence of strong approximation (see \cite[(83)]{Venkatesh}). Given an element $g_{\infty}\in\GLtwo{2}{F_{\infty}}$, there exist elements $\gamma\in\GLtwo{2}{F}$, $\kappa_\gamma\in K(\frak{c})$ such that
\begin{equation*}
\begin{pmatrix} \pi_{\frak{b}}^{-1} & \ \cr \ & \pi_{\frak{b}}^{-1} \end{pmatrix}
\begin{pmatrix} \pi_{\frak{a}}^{-1} & \ \cr \ & 1 \end{pmatrix}g_{\infty}=\gamma^{-1}
\begin{pmatrix} (\pi_{\frak{b}}^2\pi_{\frak{a}})^{-1} & \ \cr \ & 1 \end{pmatrix}g_{\infty}'\kappa_{\gamma}.
\end{equation*}
Then
\begin{equation*}
\gamma\in\GLtwo{2}{F}\cap\GLtwo{2}{F_{\infty}}
\begin{pmatrix} (\pi_{\frak{b}}^2\pi_{\frak{a}})^{-1} & \ \cr \ & 1 \end{pmatrix}
\kappa_{\gamma}\begin{pmatrix} \pi_{\frak{a}}\pi_{\frak{b}} & \ \cr \ & \pi_{\frak{b}} \end{pmatrix}.
\end{equation*}
We denote the set of such $\gamma$'s by $\Gammatr$, following \cite{Venkatesh} in notation (however, our $\Gammatr$ is not exactly the same as Venkatesh's one, because of the different normalization of the congruence subgroup $K(\frak{c})$). For $\gamma\in\Gammatr$, we denote by $\kappa_\gamma$ a corresponding element from $K(\frak{c})$. Fix an element $\gamma^*\in\Gammatr$, then $\Gammatr =\gamma^*\Gamma(\frak{a},\frak{c})=\Gamma(\frak{a}\frak{b}^2,\frak{c})\gamma^*$.

When we compute the inner product $\langle \pi_{\frak{b}}^{-1}P_2,P_1\rangle$ using the decomposition (\ref{decompositiontransition}), we see that only the $\frak{a}$-part, that is, $\pi_{\frak{b}}^{-1}P_2^{\frak{a}}$ is relevant (on the other components, at least one of $\pi_{\frak{b}}^{-1}P_2$ and $P_1$ is zero). The definition of $P_2$ and a simple computation (see \cite[(85)]{Venkatesh}) give
\begin{equation*}
\pi_{\frak{b}}^{-1}P_2^{\frak{a}}(g)= \sum_{\gamma\in\lfact{Z_{\Gamma}\Gamma_N(\frak{ab}^2,\frak{c})}{\Gammatr}}f_2(\gamma g)
\end{equation*}
by noting that in the more complicated case $r=0$, we think of the RHS in the sense of (\ref{poincareseries2}), (\ref{poincareseries2'}) and (\ref{poincareseries3}): for some $\Re\nu>1$, we have absolute convergence, then let $\nu\rightarrow 1+$ (see the Appendix).
Setting $I=\langle \pi_{\frak{b}}^{-1}P_2,P_1 \rangle$, we can unfold the integral as
\begin{equation}\label{unfoldedgeometricexpansion}
\begin{split}
I&=\int_{\lfact{Z(F_{\infty})\Gamma_N(\frak{a},\frak{c})}{\GLtwo{2}{F_{\infty}}}} \overline{f_1(g)} \sum_{\gamma\in\lfact{Z_{\Gamma}\Gamma_N(\frak{ab}^2,\frak{c})}{\Gammatr}} f_2(\gamma g)dg\\ &=\int_{\lfact{Z(F_{\infty})N(F_{\infty})}{\GLtwo{2}{F_{\infty}}}}\overline{f_1(g)} \int_{\lfact{\Gamma_N(\frak{a},\frak{c})}{N(F_{\infty}})} \overline{\psi_1(n)} \sum_{\gamma\in\lfact{Z_{\Gamma}\Gamma_N(\frak{ab}^2,\frak{c})}{\Gammatr}} f_2(\gamma ng)dndg,
\end{split}
\end{equation}
here the inner integral is essentially the Fourier coefficient of $\pi_{\frak{b}}^{-1}P_2$ corresponding to the character $\psi_1$. This unfolding is clearly valid for $r\neq 0$; while for $r=0$, we temporarily use the functions $f_1^{\nu}$, $f_2^{\nu}$, $P_1^{\nu}$, $P_2^{\nu}$, $I^{\nu}$ in the sense of (\ref{poincareseries2'}), (\ref{poincareseries3}) with any $\Re\nu>1$ (see the Appendix).

Let us split this up as $I=I_1+I_2$ according to the small and the large Bruhat cell, that is, $I_1$ is the same integral as $I$, but in the inner summation, we let $\gamma$ be upper-triangular, and $I_2$ corresponds to the rest.

First we compute $I_1$. Observe that $I_1$ is an empty integral unless $\frak{b}$ is principal. Assume then that $\frak{b}$ is generated by an
element $[\frak{b}]$. By \cite[Lemma 14]{Venkatesh},
\begin{equation*}
I_1=\sum_{\epsilon\in\rfact{\frak{o}^{\times}}{\frak{o}^{2\times}}}
\int_{\lfact{\Gamma_N(\frak{a},\frak{c})Z(F_{\infty})}{\GLtwo{2}{F_{\infty}}}} \overline{f_1(g)}f_2\left(
\begin{pmatrix}[\frak{b}]^{-1}\epsilon & \ \cr \ & [\frak{b}] \end{pmatrix}g\right)dg.
\end{equation*}
Let
\begin{equation*}
\Delta(\alpha,\alpha'[\frak{b}]^{-2})=\left\{\begin{array}{ll}1 & \mathrm{if}\
\exists\epsilon_0\in\frak{o}^{\times}: \alpha=\alpha'[\frak{b}]^{-2}\epsilon_0\mathrm{,} \cr 0 & \mathrm{otherwise.} \end{array}\right.
\end{equation*}
Take such an $\epsilon_0$ (if exists). Now let $N(F_{\infty})$ act on the left, we obtain by (\ref{naction})
\begin{equation*}
I_1=\const\Delta(\alpha,\alpha'[\frak{b}]^{-2})\norm{\frak{a}^{-1}}
\int_{\lfact{N(F_{\infty})Z(F_{\infty})}{\GLtwo{2}{F_{\infty}}}}\overline{f_1(g)} f_2\left( \begin{pmatrix}[\frak{b}]^{-1}\epsilon_0 & \ \cr \ & [\frak{b}] \end{pmatrix}g\right)dg.
\end{equation*}

Now we can turn our attention to the large Bruhat cell. First we state the explicit Bruhat decomposition.
\begin{lemm}\label{bruhat} On the large Bruhat cell, we have
\begin{equation*}
\begin{pmatrix}a & b \cr c & d\end{pmatrix}=\begin{pmatrix}1 & ac^{-1} \cr \ & 1\end{pmatrix}\begin{pmatrix} \ & 1 \cr -1 & \ \end{pmatrix} \begin{pmatrix}c & \ \cr \ &
c^{-1}(ad-bc)\end{pmatrix}\begin{pmatrix}1 & dc^{-1} \cr \ & 1\end{pmatrix}.  \qed 
\end{equation*}
\end{lemm}
For $\tau\in\lrfact{Z_{\Gamma}\Gamma_N(\frak{ab}^2,{c})}{\Gammatr}{\Gamma_N(\frak{a},\frak{c})}$, denote by $[\tau]\in\Gammatr$ any
representative. Then
\begin{equation*}
I_2=\int_{\lfact{\Gamma_N(\frak{a},\frak{c})Z(F_{\infty})}{\GLtwo{2}{F_{\infty}}}}
\sum_{\substack{\tau\in\lrfact{Z_{\Gamma}\Gamma_N(\frak{ab}^2,\frak{c})}{\Gammatr}{\Gamma_N(\frak{a},\frak{c})}\\ [\tau]\notin B(F_{\infty})}}
\sum_{\mu\in\Gamma_N(\frak{a},\frak{c})}\overline{f_1(g)}f_2([\tau]\mu g)dg.
\end{equation*}
Now folding together the integral and the $\mu$-sum, we obtain
\begin{equation*}
I_2=\sum_{\substack{\tau\in\lrfact{Z_{\Gamma}\Gamma_N(\frak{ab}^2,\frak{c})}{\Gammatr}{\Gamma_N(\frak{a},\frak{c})}\\ [\tau]\notin B(F_{\infty})}}
\int_{\lfact{Z(F_{\infty})}{\GLtwo{2}{F_{\infty}}}}\overline{f_1(g)}f_2([\tau]g)dg.
\end{equation*}
Let $[\tau]=n_{1,[\tau]}\weyl a_{[\tau]}n_{2,[\tau]}$ according to the Bruhat decomposition. Then by (\ref{naction}),
\begin{equation*}
\begin{split}
I_2&=\sum_{\substack{\tau\in\lrfact{Z_{\Gamma}\Gamma_N(\frak{ab}^2,\frak{c})}{\Gammatr}{\Gamma_N(\frak{a},\frak{c})}\\ [\tau]\notin B(F_{\infty})}}
\psi_1(n_{2,[\tau]})\psi_2(n_{1,[\tau]}) \cdot\int_{\lfact{Z(F_{\infty})}{\GLtwo{2}{F_{\infty}}}}\overline{f_1(g)}f_2(\weyl a_{[\tau]}g)dg\\
&=\sum_{\substack{\tau\in\lrfact{Z_{\Gamma}\Gamma_N(\frak{ab}^2,\frak{c})}{\Gammatr}{\Gamma_N(\frak{a},\frak{c})}\\ [\tau]\notin B(F_{\infty})}}
\psi_1(n_{2,[\tau]})\psi_2(n_{1,[\tau]})\\ &\qquad\qquad\qquad\qquad\qquad\qquad\qquad \cdot\int_{\lfact{N(F_{\infty})Z(F_{\infty})}{\GLtwo{2}{F_{\infty}}}}\overline{f_1(g)} \int_{N(F_{\infty})}\overline{\psi_1(n)}f_2(\weyl a_{[\tau]}ng)dndg.
\end{split}
\end{equation*}
By Lemma \ref{bruhat},
\begin{equation*}
n_{2,[\tau]}=\begin{pmatrix}1 & dc^{-1} \cr \ & 1 \end{pmatrix},\qquad n_{1,[\tau]}=\begin{pmatrix}1 & ac^{-1} \cr \ & 1 \end{pmatrix}.
\end{equation*}
Putting everything together and using again \cite[Lemma 14]{Venkatesh} for the explicit description of $\Gammatr$ (keeping in mind that our normalization differs a little), we obtain
\begin{equation}\label{geometricside}
\begin{split}
I&=\const\Delta(\alpha,\alpha'[\frak{b}]^{-2})\norm{\frak{a}^{-1}}
\int_{\lfact{N(F_{\infty})Z(F_{\infty})}{\GLtwo{2}{F_{\infty}}}}\overline{f_1(g)} f_2\left(\begin{pmatrix}[\frak{b}]^{-1}\epsilon_0 & \ \cr \ & [\frak{b}] \end{pmatrix}g\right)dg\\
&+\const\sum_{c\in \frak{abc},\epsilon\in\rfact{\frak{o}^{\times}}{\frak{o}^{2\times}}}
KS(\epsilon\alpha,\frak{a}^{-1}\frak{d}^{-1};\alpha', \frak{a}^{-1}\frak{b}^{-2}\frak{d}^{-1};c,\frak{a}^{-1}\frak{b}^{-1}\frak{d}^{-1})\\ &\cdot\int_{\lfact{N(F_{\infty})Z(F_{\infty})}{\GLtwo{2}{F_{\infty}}}}\overline{f_1(g)} \int_{N(F_{\infty})}\overline{\psi_1(n)}f_2\left(\weyl\begin{pmatrix}1 & \ \cr \ & \epsilon c^{-2}
\end{pmatrix}ng\right)dndg.
\end{split}
\end{equation}

Now in case of $r=0$, we may move to the limit $\nu=1$ by analytic continuation: the LHS made sense for $\nu=1$ even in the beginning; while each integral on the RHS is absolutely convergent even for $\nu=1$ (see the Appendix, and also the forthcoming computation on the Jacquet transform of $f_2$, from which the needed magnitude will be trivial).

\subsection{Spectral description}\label{Spectral description}

On the other hand, we decompose spectrally the inner product. With this aim in mind, take an orthonormal system $\mathbf{B}_{\FS}$ in the cuspidal spectrum, which consists of pure weight forms $\mathbf{f}$ and moreover assume that $Z(\AAA)$ acts on each element $\mathbf{f}\in\mathbf{B}_{\FS}$ via multiplication by some class group character $\omega_{\mathbf{f}}$ ($\mathbf{B}_{\FS}$ is larger than $\mathbf{B}$ introduced in Section \ref{The Kuznetsov formula}, see \ref{Transition between adelic and classical forms}). Observe that if $K(\frak{c})$ acts nontrivially on some $\mathbf{f}\in\mathbf{B}_{\FS}$, then $\mathbf{f}$ is orthogonal to our Poincar\'e series. Therefore from now on, we can restrict the cuspidal summation to $\mathbf{B}_{\FS}(\frak{c})$. We have
\begin{equation*}
\langle\pi_{\frak{b}}^{-1}P_2,P_1\rangle =\sum_{\mathbf{f}\in\mathbf{B}_{\FS}(\frak{c})}\overline{\langle P_1,\mathbf{f}\rangle}\langle
\pi_{\frak{b}}^{-1}P_2,\mathbf{f}\rangle +\int_{0}^{\infty}\bigoplus_{\chi=|\cdot|^{iy}\ \mathrm{on}\ F_{\infty,+}^{\mathrm{diag}}}
\langle\mathrm{pr}_{V_{\chi}}P_1,\mathrm{pr}_{V_{\chi}}\pi_{\frak{b}}^{-1} P_2\rangle_{V_{\chi}} d^{\times}y.
\end{equation*}
Here $\mathrm{pr}_{V_{\chi}}$ denotes the $\chi$-part of the spectral decomposition, see \cite{GelbartJacquet}. The discussion of the contribution of the continuous spectrum $CSC$ is analogous to the discrete spectrum which we present below. We note that the contribution of discrete but non-cuspidal representations is zero (this can be easily seen from their Fourier-Whittaker expansion).

As for the cuspidal part,
\begin{equation*}
\begin{split}
\overline{\langle P_1,\mathbf{f}\rangle}&= \left[K(\frak{o}):K(\frak{c})\right]^{-\frac{1}{2}}
\int_{\lfact{Z(F_{\infty})\Gamma_N(\frak{a},\frak{c})}{\GLtwo{2}{F_{\infty}}}} \overline{f_1(g)}\mathbf{f}^{\frak{a}}(g)\\
&=\left[K(\frak{o}):K(\frak{c})\right]^{-\frac{1}{2}} \int_K\int_A\overline{f_1\left(\begin{pmatrix}y & \ \cr \ & 1\end{pmatrix}k\right)}y^{-2\rho}\\ &\cdot
\int_{\lfact{\Gamma_N(\frak{a},\frak{c})}{N(F_{\infty})}} \psi_1(-n)\mathbf{f}^{\frak{a}}\left(\begin{pmatrix}1 & n \cr \ &
1\end{pmatrix}\begin{pmatrix}y & \ \cr \ & 1\end{pmatrix}k\right)dnd^{\times}ydk,
\end{split}
\end{equation*}
where the quasifactor $\rho_j$ is defined to be $\frac{1}{2}\deg(F_{\infty,j}:\RR)$, and the factor $\left[K(\frak{o}):K(\frak{c})\right]^{-\frac{1}{2}}$ is explained in Lemma \ref{measuretransition}.

Using the $N(F_{\infty})$- and $K(F_{\infty})$-transformation properties and the normalization of Fourier coeffecients, we see
\begin{equation*}
\begin{split}
\overline{\langle P_1,\mathbf{f}\rangle}=&\ \const\left[K(\frak{o}):K(\frak{c})\right]^{-\frac{1}{2}} \lambda_{\mathbf{f}}(\alpha\frak{a}^{-1}) \frac{\sqrt{\mathcal{N}(\frak{a}^{-1})}}{\sqrt{\mathcal{N}(\alpha)}}\\
&\prod_{j=1}^r\left\{\int_0^{\infty} \overline{f_1\left(\begin{pmatrix} y_j & \ \cr \ & 1\end{pmatrix}\right)}\whit_j^{(q_j)}(|\alpha_j|y_j)|y_j|^{-1}d^{\times}y_j\right.\\ &\left.\qquad \mathrm{OR}\ (-1)^{\varepsilon_j}\int_0^{\infty}\overline{f_1\left(\begin{pmatrix} y_j & \ \cr \ & 1\end{pmatrix}\right)}\whit_j^{(-q_j)}(|\alpha_j|y_j)|y_j|^{-1}d^{\times}y_j\right\}\\ &\prod_{j=r+1}^{r+s}\int_0^{\infty}\overline{f_1\left(\begin{pmatrix} y_j & \ \cr \ & 1\end{pmatrix}\right)}\whit_j(\alpha_jy_j)|y_j|^{-2}d^{\times}y_j,
\end{split}
\end{equation*}
by the 'OR' we mean the dependence on the real weight of $\mathbf{f}$, the superscripts $(q_j)$ and $(-q_j)$ denotes that we use the Whittaker function corresponding either to weight $q_j$ or to weight $-q_j$ (recall the observation at the end of Section \ref{Building blocks at real places}); of course, when summing over $\mathbf{B}_{\FS}$, we need both. Again, when $r=0$, the RHS will turn out to make sense for $\nu=1$, so we can move to the limit (see the Appendix and the archimedean computations).

Then similarly computing $\langle \pi_{\frak{b}}^{-1}P_2, \mathbf{f} \rangle$, we obtain
\begin{equation}\label{spectralside}
\begin{split}
\langle\pi_{\frak{b}}^{-1}&P_2,P_1\rangle= \const\left[K(\frak{o}):K(\frak{c})\right]^{-1} \frac{\mathcal{N}(\frak{a}^{-1} \frak{b}^{-1})}{\sqrt{\mathcal{N}(\alpha\alpha')}} \sum_{\mathbf{f}\in\mathbf{B}_{\FS}(\frak{c})} \omega_{\mathbf{f}}(\pi_{\frak{b}}^{-1}) \lambda_{\mathbf{f}}(\alpha\frak{a}^{-1}) \overline{\lambda_{\mathbf{f}}(\alpha'\frak{a}^{-1}\frak{b}^{-2})}\\ &\prod_{j=1}^r\left(\int_0^{\infty} \overline{f_1\left(\begin{pmatrix} y_j & \ \cr \ & 1\end{pmatrix}\right)}\whit_j^{(q_j)}(\alpha_jy_j)|y_j|^{-1}d^{\times}y_j \int_0^{\infty} f_2\left(\begin{pmatrix} y_j & \ \cr \ & 1\end{pmatrix}\right)\overline{\whit_j^{(q_j)}(\alpha_j'y_j)}|y_j|^{-1} d^{\times}y_j\right.\\ &\left.\qquad +\int_0^{\infty} \overline{f_1\left(\begin{pmatrix} y_j & \ \cr \ & 1\end{pmatrix}\right)}\whit_j^{(-q_j)}(\alpha_jy_j)|y_j|^{-1}d^{\times}y_j \int_0^{\infty} f_2\left(\begin{pmatrix} y_j & \ \cr \ & 1\end{pmatrix}\right)\overline{\whit_j^{(-q_j)}(\alpha_j'y_j)}|y_j|^{-1} d^{\times}y_j\right)\\
&\prod_{j=r+1}^{r+s}\int_0^{\infty}\overline{f_1\left(\begin{pmatrix} y_j & \ \cr \ & 1\end{pmatrix}\right)}\whit_j(\alpha_jy_j)|y_j|^{-2}d^{\times}y_j \int_0^{\infty} f_2\left(\begin{pmatrix} y_j & \ \cr \ & 1\end{pmatrix}\right)\overline{\whit_j(\alpha_j'y_j)}|y_j|^{-2}d^{\times}y_j+CSC.
\end{split}
\end{equation}

\section{Archimedean computations}

In this section, we compute the local contributions to integrals given in Section \ref{Geometric description} and Section \ref{Spectral description} of functions defined in Section \ref{Building blocks at real places} and Section \ref{Building blocks at complex places}.

\subsection{The real case}

We introduce the following notation: for any $g$, let
\begin{equation*}
g^*=\begin{pmatrix}-1 & \ \cr \ & 1\end{pmatrix} g \begin{pmatrix}-1 & \ \cr \ & 1\end{pmatrix}.
\end{equation*}

We start with the evaluation of the geometric side (\ref{geometricside}). For the small Bruhat cell, we need to investigate the function
\begin{equation*}
\widetilde{\mathcal{L}}_q^{\alpha_j'}\theta\left( \begin{pmatrix} [\frak{b}]_j^{-2}\epsilon_{0_j} & \ \cr \ & 1\end{pmatrix} g\right).
\end{equation*}
In the case of $\epsilon_{0_j}<0$, we can rewrite this as
\begin{equation*}
\widetilde{\mathcal{L}}_q^{\alpha_j'}\theta\left( \begin{pmatrix} [\frak{b}]_j^{-2}|\epsilon_{0_j}| & \ \cr \ & 1\end{pmatrix} g^* \begin{pmatrix}-1 & \ \cr \ & 1 \end{pmatrix}\right)= \widetilde{\mathcal{L}}_q^{-\alpha_j'}\theta\left( \begin{pmatrix} [\frak{b}]_j^{-2}|\epsilon_{0_j}| & \ \cr \ & 1\end{pmatrix} g^*\right).
\end{equation*}
This is of weight $-q$. Assuming $q\neq 0$, we see that integrating this against $\widetilde{\mathcal{L}}_q^{\alpha_j}$ on $\lfact{N(\RR)Z(\RR)}{\GLtwo{2}{\RR}}$, we obtain $0$. So this term vanishes unless $\epsilon_{0_j}>0$ (under the condition $q\neq 0$). If $\epsilon_{0_j}>0$, we have
\begin{equation*}
\widetilde{\mathcal{L}}_q^{\alpha_j'}\theta\left( \begin{pmatrix} [\frak{b}]_j^{-2}\epsilon_{0_j} & \ \cr \ & 1\end{pmatrix} g\right)=\widetilde{\mathcal{L}}_q^{\alpha_j}\theta(g).
\end{equation*}
Now we may apply
\cite[Corollary 3.6]{BruggemanMiatello} by noting
\begin{equation}\label{slglintegraltransition}
\langle \widetilde{\mathcal{L}}_q^{\alpha}\theta, \widetilde{\mathcal{L}}_q^{\alpha}\eta \rangle_{\lfact{N(\RR)Z(\RR)}{\GLtwo{2}{\RR}}}= \langle \widetilde{\mathcal{L}}_q^{\alpha}\theta, \widetilde{\mathcal{L}}_q^{\alpha}\eta \rangle_{\lfact{N(\RR)}{Z(\RR)\GLtwo{2}{\RR}^+}} +\langle \widetilde{\mathcal{L}}_q^{-\alpha}\theta, \widetilde{\mathcal{L}}_q^{-\alpha}\eta \rangle_{\lfact{N(\RR)}{Z(\RR)\GLtwo{2}{\RR}^+}},
\end{equation}
where $\GLone{2}^+(\RR)$ stands for the elements of $\GLtwo{2}{\RR}$ with positive determinant. We obtain
\begin{equation}\label{scalarproductofrealblocks}
\begin{split}
\langle \widetilde{\mathcal{L}}_q^{\alpha_j}\theta, \widetilde{\mathcal{L}}_q^{\alpha_j}\eta &\rangle_{\lfact{N(\RR)Z(\RR)}{\GLtwo{2}{\RR}}}= \const|\alpha_j| \left(\int_{0}^{i\infty}\theta(\nu)\overline{\eta(\nu)} \sum_{\pm}\left|\frac{\Gamma\left(\frac{1}{2}+ \nu\pm\frac{\sign(\alpha_j)q}{2}\right)}{ \Gamma(2\nu)}\right|^2\frac{d\nu}{2\pi i}\right.\\ &\left.+\sum_{1<b\leq|\sign(\alpha_j)q|} \theta\left(\frac{b-1}{2}\right) \overline{\eta\left(\frac{b-1}{2}\right)} \sum_{\pm}\frac{b-1}{\left(\frac{\pm\sign(\alpha_j)q-b}{2}\right)! \left(\frac{\pm\sign(\alpha_j)q+b-2}{2}\right)!}\right).
\end{split}
\end{equation}
By the further notation
\begin{equation*}
\lambda(\nu,q)=\sum_{\pm}\frac{1}{\Gamma \left(\frac{1}{2}-\nu\pm\frac{\sign(\alpha_j)q}{2} \right) \Gamma \left(\frac{1}{2}+\nu\pm\frac{\sign(\alpha_j')q}{2} \right)},
\end{equation*}
this equals
\begin{equation*}
\begin{split}
&\const |\alpha_j| \left(
\int_0^{i\infty}\overline{\eta(\nu)}\theta(\nu) \lambda(\nu,q)(-4\pi\nu)\tan(\pi\nu) \frac{d\nu}{2\pi i}\right.\\ &\left.+ \sum_{1<b\leq|\sign(\alpha_j)q|}\overline{\eta\left(\frac{b-1}{2}\right)} \theta\left(\frac{b-1}{2}\right) \lambda\left(\frac{b-1}{2},q\right)\right).
\end{split}
\end{equation*}

On the large Bruhat cell, we need to compute
\begin{equation*}
\int_{-\infty}^{\infty} e^{-2\pi i\alpha_jx} \widetilde{\mathcal{L}}_q^{\alpha_j'}\theta \left(\begin{pmatrix}\ & 1 \cr -1 & \ \end{pmatrix} \begin{pmatrix}1 & \ \cr \ & c_j^{-2}\epsilon_j \end{pmatrix} \begin{pmatrix}1 & x \cr \ & 1 \end{pmatrix} g\right)dx.
\end{equation*}
Again, if $\epsilon_j<0$, then
\begin{equation*}
\begin{split}
\int_{\lfact{N(\RR)Z(\RR)}{\GLtwo{2}{\RR}}} \overline{(\widetilde{\mathcal{L}}_q^{\alpha}\eta)(g)} &\int_{-\infty}^{\infty} e^{-2\pi i\alpha_jx} \widetilde{\mathcal{L}}_q^{\alpha_j'}\theta \left(\begin{pmatrix}\ & 1 \cr -1 & \ \end{pmatrix} \begin{pmatrix}1 & \ \cr \ & c_j^{-2}\epsilon_j \end{pmatrix} \begin{pmatrix}1 & x \cr \ & 1 \end{pmatrix} g\right)dxdg=\\ &\int_{\lfact{Z(\RR)}{\GLtwo{2}{\RR}}} \overline{(\widetilde{\mathcal{L}}_q^{\alpha_j}\eta)(g)} \widetilde{\mathcal{L}}_q^{\alpha_j'}\theta \left(\begin{pmatrix}\ & 1 \cr -1 & \ \end{pmatrix} \begin{pmatrix}1 & \ \cr \ & c_j^{-2}\epsilon_j \end{pmatrix}g\right)dg=0,
\end{split}
\end{equation*}
as we integrate a weight $q$ function against a weight $-q$ function like before (assuming again $q\neq 0$). So we may assume $\epsilon_j>0$. If $\det g>0$, by \cite[Theorem 3.8, (3.34-35)]{BruggemanMiatello}, we obtain that this integral equals
\begin{equation*}
\int_{-\infty}^{\infty} e^{-2\pi i\alpha_jx} \widetilde{\mathcal{L}}_q^{\alpha_j'}\theta \left(\begin{pmatrix}\ & 1 \cr -1 & \ \end{pmatrix} \begin{pmatrix}1 & \ \cr \ & c_j^{-2}\epsilon_j \end{pmatrix} \begin{pmatrix}1 & x \cr \ & 1 \end{pmatrix} g\right)dx=
\widetilde{\mathcal{L}}_q^{\alpha_j}\widetilde{\theta}(g)
\end{equation*}
with
\begin{equation*}
\widetilde{\theta}(\nu)=\theta(\nu)\frac{1}{2|\alpha_j|} \frac{\Gamma\left(\frac{1}{2}+\nu+\frac{\sign(\alpha_j)q}{2}\right)}{\Gamma \left(\frac{1}{2}+\nu+\frac{\sign(\alpha_j')q}{2} \right)} \mathcal{B}_{\nu} \left(4\pi\frac{|\alpha_j\alpha_j'\epsilon_j|^{\frac{1}{2}}}{|c_j|}\right) \frac{|\alpha_j\alpha_j'\epsilon_j|^{\frac{1}{2}}}{|c_j|}
\end{equation*}
with $\mathcal{B}$ defined in \ref{archimedeanbesseltransformsandmeasures}.
The holomorphy condition of \cite[Theorem 3.8]{BruggemanMiatello} is satisfied by our condition $\sign(\alpha_j)=\sign(\alpha'_j)$. The case $\det g<0$ can be reduced as before (use (\ref{slglintegraltransition}) again). By (\ref{scalarproductofrealblocks}), up to some constant, the contribution of the large cell is
\begin{equation}\label{contributionoflargecellreal}
\begin{split}
\frac{|\alpha_j\alpha_j'\epsilon_j|^{\frac{1}{2}}}{|c_j|}
&\left(\int_0^{i\infty}\overline{\eta(\nu)}\theta(\nu) \lambda(\nu,q)(-4\pi\nu)\tan(\pi\nu) \mathcal{B}_{\nu} \left(4\pi\frac{|\alpha_j\alpha_j'\epsilon_j|^{\frac{1}{2}}}{|c_j|}\right)  \frac{d\nu}{2\pi i}\right.\\ &\left.+ \sum_{1<b\leq\sign(\alpha_j)q} \overline{\eta\left(\frac{b-1}{2}\right)} \theta\left(\frac{b-1}{2}\right) \lambda\left(\frac{b-1}{2},q\right) \mathcal{B}_{\frac{b-1}{2}} \left(4\pi\frac{|\alpha_j\alpha_j' \epsilon_j|^{\frac{1}{2}}}{|c_j|}\right)\right).
\end{split}
\end{equation}

On the spectral side (\ref{spectralside}), we need to integrate our building block against our normalized Whittaker function. Using \cite[Corollary 3.5]{BruggemanMiatello},
\begin{equation*}
\begin{split}
\int_0^{\infty} \widetilde{\mathcal{L}}_q^{\alpha_j}\eta(y)&\overline{\whit_j(\alpha_j y)}|y|^{-1}d^{\times}y=\\ &\const|\alpha_j|\eta(\nu) \frac{-i^{\sign(\alpha_j)\frac{q}{2}}}{\{\Gamma(1/2-\nu_j+\sign(\alpha_j)q_j/2) \Gamma(1/2+\nu_j+\sign(\alpha_j)q_j/2)\}^{1/2}}.
\end{split}
\end{equation*}
As the inner product of such two, at a real place in (\ref{spectralside}) we obtain
\begin{equation}\label{spectralsidereal}
\begin{split}
&\int_0^{\infty} \overline{\widetilde{\mathcal{L}}_q^{\alpha_j}\eta(y)}\whit_j^{(q)}(\alpha_jy) |y|^{-1}d^{\times}y \int_0^{\infty} \widetilde{\mathcal{L}}_q^{\alpha_j'}\theta(y) \overline{\whit_j^{(q)}(\alpha_j'y)}|y|^{-1}d^{\times}y\\ 
&+\int_0^{\infty} \overline{\widetilde{\mathcal{L}}_q^{\alpha_j}\eta(y)}\whit_j^{(-q)}(\alpha_jy) |y|^{-1}d^{\times}y \int_0^{\infty} \widetilde{\mathcal{L}}_q^{\alpha_j}\theta(y) \overline{\whit_j^{(-q_j)}(\alpha_j'y)}|y|^{-1}d^{\times}y=\\ &\const|\alpha_j\alpha_j'|\overline{\eta(\nu)}\theta(\nu)\lambda(\nu,q).
\end{split}
\end{equation}

\subsection{The complex case}

We execute the complex analog of the above procedure. On the small cell, in our normalization \cite[display between (10.22-23)]{Lokvenec} gives 
\begin{equation*}
\widetilde{\mathcal{L}}_{l,q}^{\alpha_j'}\theta\left(\begin{pmatrix} [\frak{b}]_j^{-2}\epsilon_{0_j} & \ \cr \ & 1\end{pmatrix} g\right)=\widetilde{\mathcal{L}}_{l,q}^{\alpha_j}\theta(g).
\end{equation*}
We have to integrate this against $\widetilde{\mathcal{L}}_{l,q}^{\alpha_j}\eta$. Using \cite[Lemma 9.1.5]{Lokvenec},
\begin{equation*}
\langle \widetilde{\mathcal{L}}_{l,q}^{\alpha_j}\theta, \widetilde{\mathcal{L}}_{l,q}^{\alpha_j}\eta\rangle_{\lfact{N}{G}}=\const  |\alpha_j|^2 \sum_{|p|\leq l}\int_{(0)}\overline{\eta(\nu,p)}\theta(\nu,p) \Gamma(l+1-\nu)\Gamma(l+1+\nu) \frac{\sin^2\pi(\nu-p)}{p^2-\nu^2}\nu^{2\epsilon(p)}d\nu
\end{equation*}
with $\epsilon(0)=1$, $\epsilon(p)=-1$ for $p\in\ZZ\setminus\{0\}$. Introducing \begin{equation*}
\lambda_l(\nu,p)=\Gamma(l+1-\nu)\Gamma(l+1+\nu) \frac{\sin^2\pi(\nu-p)}{(p^2-\nu^2)^2}\nu^{2\epsilon(p)},
\end{equation*}
we get
\begin{equation}\label{scalarproductofcomplexblocks}
\langle \widetilde{\mathcal{L}}_{l,q}^{\alpha_j}\theta, \widetilde{\mathcal{L}}_{l,q}^{\alpha_j}\eta\rangle_{\lfact{N}{G}}=
\const |\alpha_j|^2 \sum_{|p|\leq l}\int_{(0)}\overline{\eta(\nu,p)}\theta(\nu,p)\lambda_l(\nu,p)(p^2-\nu^2)d\nu.
\end{equation}
Note that $\lambda_l(\nu,p)$ is nonzero, if $l$ is large enough.

On the large Bruhat cell, the corresponding integral is
\begin{equation*}
\int_{\CC} e^{-2\pi i\alpha_j(x+\overline{x})} \widetilde{\mathcal{L}}_{l,q}^{\alpha_j'}\theta \left(\begin{pmatrix}\ & 1 \cr -1 & \ \end{pmatrix} \begin{pmatrix}1 & \ \cr \ & c_j^{-2}\epsilon_j \end{pmatrix} \begin{pmatrix}1 & x \cr \ & 1 \end{pmatrix} g\right)dx.
\end{equation*}
Now using \cite[Lemma 9.1.8]{Lokvenec}, we obtain
\begin{equation*}
\const\left|\frac{\alpha'_j\epsilon_j}{\alpha_jc_j^2}\right| \widetilde{\mathcal{L}}_{l,q}^{\alpha_j} \left(\mathcal{B}_{\nu,p} \left(4\pi\frac{(\alpha_j\alpha_j'\epsilon_j)^\frac{1}{2}}{c_j}\right) \theta(\nu,p)\right)(g),
\end{equation*}
where
\begin{equation*}
\mathcal{B}_{\nu,p}(z)= \frac{1}{\sin\pi(\nu-p)} \{|z/2|^{-2\nu}(iz/|z|)^{2p}\mathcal{J}^*_{-\nu,-p}(z)- |z/2|^{2\nu}(iz/|z|)^{-2p}\mathcal{J}^*_{\nu,p}(z)\},
\end{equation*}
with $\mathcal{J}^*$ defined in Lemma \ref{jacquettransformofgoodmanwallach}.
Now by (\ref{scalarproductofcomplexblocks}),
\begin{equation}\label{contributionoflargecellcomplex}
\const\ \left|\frac{\alpha_j\alpha_j'\epsilon_j}{c_j^2}\right| \sum_{|p|\leq l}\int_{(0)}\overline{\eta(\nu,p)}\theta(\nu,p)\lambda_l(\nu,p) \mathcal{B}_{(\nu,p)} \left(4\pi\frac{(\alpha_j\alpha_j'\epsilon_j)^\frac{1}{2}}{c_j}\right) (p^2-\nu^2)d\nu.
\end{equation}

We are left to work with the spectral side. To deliver the computation at a complex place of (\ref{spectralside}), we use \cite[(10.4-7)]{Lokvenec}, obtaining
\begin{equation*}
\begin{split}
&\int_0^{\infty}\widetilde{\mathcal{L}}_{l,q}^{\alpha_j}\eta \left(\begin{pmatrix}y & \ \cr \ & 1\end{pmatrix}\right) \overline{\whit_j(\alpha_jy)}|y|^{-2}d^{\times}y=\\ &\const |\alpha_j|^2i^p\Gamma(l+1-\overline\nu) \frac{\sin \pi(\overline{\nu}-p)}{\overline{\nu}^2-p^2} \overline{\nu}^{\varepsilon(p)}\eta(-\overline{\nu},p) \sqrt{\left|\frac{\Gamma(l+1+\nu)}{\Gamma(l+1-\nu)}\right|}.
\end{split}
\end{equation*}
Note that the last factor is $1$, unless we are in the complementary series and in this case, $p=0$. Now taking the inner product of such two, we obtain
\begin{equation}\label{spectralsidecomplex}
\begin{split}
&\int_0^{\infty} \overline{\widetilde{\mathcal{L}}_{l,q}^{\alpha_j}\eta(y)} \whit_j(\alpha_jy)|y|^{-2}d^{\times}y \int_0^{\infty} \widetilde{\mathcal{L}}_{l,q}^{\alpha_j'}\theta(y) \overline{\whit_j(\alpha_j'y)}|y|^{-2}d^{\times}y=\\ &=\const |\alpha_j\alpha'_j|^2\overline{\eta(\nu,p)}\theta(\nu,p)\lambda_l(\nu,p).
\end{split}
\end{equation}
Observe the similar behaviour of the real and the complex case, even the factors coming from $\alpha,\alpha'$ are the same by noting that for the complex modulus $|\cdot|_{\CC}$, $|z|_{\CC}=|z|^2$. Of course, the applied integral transforms and hence the integral kernels in the final formulas show some difference.

We see that we computed all integrals converge here as we promised in Section \ref{Geometric description} and Section \ref{Spectral description}.

\section{Derivation of the sum formula}

\subsection{Preliminary sum formulas}

In this section we state some preliminary versions of the sum formula.
\begin{lemm} With the above notation, assuming that $\alpha\alpha'$ is totally positive,
\begin{equation}\label{specialformula}
\begin{split}
&\left[K(\frak{o}):K(\frak{c})\right]^{-1} \sum_{\mathbf{f}\in\mathbf{B}_{\FS}(\frak{c})} \omega_{\mathbf{f}}(\pi_{\frak{b}}^{-1})(\overline{\eta}\theta\lambda)  \lambda_{\mathbf{f}}(\alpha\frak{a}^{-1}) \overline{\lambda_{\mathbf{f}}(\alpha'\frak{a}^{-1}\frak{b}^{-2})}+CSC=\\
&\const \Delta(\alpha,\alpha'[\frak{b}]^{-2}) \int(\overline{\eta}\theta\lambda) d\mu+\\ &\const \sum_{c\in \frak{abc}}\sum_{\epsilon\in\rfact{\frak{o}_+^{\times}}{\frak{o}^{2\times}}} \frac{KS(\epsilon\alpha,\frak{a}^{-1}\frak{d}^{-1};\alpha', \frak{a}^{-1}\frak{b}^{-2}\frak{d}^{-1};c,\frak{a}^{-1}\frak{b}^{-1}\frak{d}^{-1}) }{\norm{c\frak{a}^{-1}\frak{b}^{-1}}}
\int\mathcal{B}_{(\nu,p)}\left(\frac{|\alpha\alpha'\epsilon|^{\frac{1}{2}}}{c}\right) (\overline{\eta}\theta\lambda) d\mu.
\end{split}
\end{equation}
\end{lemm}
\begin{proof} Immediate from (\ref{geometricside}), (\ref{spectralside}), (\ref{scalarproductofrealblocks}), (\ref{contributionoflargecellreal}), (\ref{spectralsidereal}), (\ref{scalarproductofcomplexblocks}), (\ref{contributionoflargecellcomplex}) and (\ref{spectralsidecomplex}).
\end{proof}

In an actual application, some ideals $\frak{a}^{-1},\frak{a}'^{-1}$ are given. If there is some ideal $\frak{b}$ such that $\frak{a}'^{-1}$ equals $\frak{a}^{-1}\frak{b}^{-2}$ up to a totally positive principal ideal, that is, $\frak{a}\frak{a}'^{-1}$ is a square in the narrow class group, then adjusting $\alpha'$ in (\ref{specialformula}), we obtain a formula including $\lambda_{\mathbf{f}}(\alpha\frak{a}^{-1})\overline{\lambda_{\mathbf{f}}(\alpha'\frak{a}'^{-1})}$. Denote by $C$ a fixed set of narrow class representatives $\frak{m}$ for which $\frak{m}^2\frak{aa}'^{-1}$ is a principal ideal generated by a totally positive element $\gamma_{\frak{m}}$, fixed once for all, and let $C'$ be a set of representatives for the rest of ideals.

\begin{lemm} For all $\frak{m}\in C$ and $\alpha\in\frak{ad}^{-1}$, $\alpha\in\frak{a}'\frak{d}^{-1}$ such that $\alpha\alpha'$ is totally positive, we have
\begin{equation}\label{preliminarysumformula'}
\begin{split}
&\left[K(\frak{o}):K(\frak{c})\right]^{-1} \sum_{\mathbf{f}\in\mathbf{B}_{\FS}(\frak{c})} \omega_{\mathbf{f}}(\pi_{\frak{m}}^{-1})(\overline{\eta}\theta\lambda) \lambda_{\mathbf{f}}(\alpha\frak{a}^{-1}) \overline{\lambda_{\mathbf{f}}(\alpha'\frak{a}'^{-1})}+CSC=\\
&\const \Delta(\alpha\frak{a}^{-1},\alpha'\frak{a}'^{-1}) \int(\overline{\eta}\theta\lambda) d\mu+\\ &\const \sum_{c\in \frak{amc}}\sum_{\epsilon\in\rfact{\frak{o}_+^{\times}}{\frak{o}^{2\times}}} \frac{KS(\epsilon\alpha,\frak{a}^{-1}\frak{d}^{-1};\alpha'\gamma_{\frak{m}}, \frak{a}'^{-1}\frak{d}^{-1};c,\frak{a}^{-1}\frak{m}^{-1}\frak{d}^{-1}) }{\norm{c\frak{a}^{-1}\frak{m}^{-1}}}
\int\mathcal{B}_{(\nu,p)}\left(\frac{|\alpha\alpha'\gamma_\frak{m} \epsilon|^{\frac{1}{2}}}{c}\right) (\overline{\eta}\theta\lambda) d\mu. \qed
\end{split}
\end{equation}
\end{lemm}
Observe that
\begin{equation*}
\sum_{\mathbf{f}\in\mathbf{B}_{\FS}(\frak{c})} \omega_{\mathbf{f}}(\pi_{\frak{m}}^{-1})(\overline{\eta}\theta\lambda) \lambda_{\mathbf{f}}(\alpha\frak{a}^{-1}) \overline{\lambda_{\mathbf{f}}(\alpha'\frak{a}'^{-1})}
\end{equation*}
is independent of the choice of the basis $\mathbf{B}_{\FS}(\frak{c})$. Indeed, $\lambda_{\mathbf{f}}$ depends linearly on $\mathbf{f}$, which shows that under orthogonal transformations, the sum is invariant. Note that $\mathbf{B}_{\FS}(\frak{c})$ is the disjoint union of the bases of subspaces corresponding to fixed archimedean parameters ($q,(l,q),\nu,p$): we apply an orthogonal transformation inside in each such subspace (of fixed archimedean parameters), leaving the factor $\overline{\eta}\theta\lambda$ invariant. In particular, this holds for the orthonormal basis $\mathbf{B}_{\FS}(\frak{c})\otimes\chi= \{\mathbf{f}\otimes\chi:\mathbf{f}\in\mathbf{B}_{\FS}(\frak{c})\}$, where $\chi$ is a character of the narrow class group. This is indeed a basis, since the central character is multiplied by $\chi^2$, which is trivial on the archimedean ideles. Also note that the archimedean parameters ($q,(l,q),\nu,p$) are invariant under these twists.

\begin{lemm}[\cite{Harcospriv}]\label{harcos} Let $\frak{m}$ be a narrow class representative from either $C$ or $C'$. Denote by $\Xi$ the dual of the narrow class group. Assume $\alpha\in\frak{ad}^{-1}$, $\alpha'\in\frak{a}'\frak{d}^{-1}$ such that $\alpha\alpha'$ is totally positive. Then
\begin{equation}\label{average}
\sum_{\mathbf{f}\in\mathbf{B}_{\FS}(\frak{c})}\omega_{\mathbf{f}}(\pi_{\frak{m}}^{-1}) (\overline{\eta}\theta\lambda) \lambda_{\mathbf{f}}(\alpha\frak{a}^{-1}) \overline{\lambda_{\mathbf{f}}(\alpha'\frak{a}'^{-1})}=
\frac{1}{|\Xi|}\sum_{\chi\in\Xi} \sum_{\mathbf{f}\in\mathbf{B}_{\FS}(\frak{c})}\omega_{\mathbf{f}\otimes{\chi}}(\pi_{\frak{m}}^{-1}) (\overline{\eta}\theta\lambda) \lambda_{\mathbf{f}\otimes{\chi}}(\alpha\frak{a}^{-1}) \overline{\lambda_{\mathbf{f}\otimes{\chi}}(\alpha'\frak{a}'^{-1})}.
\end{equation}
The analogous identity holds for $CSC$. Moreover, if $\frak{m}\in C'$, the sum is $0$ (and so is $CSC$).
\end{lemm}
\begin{proof} If $\mathbf{f}\in\mathbf{B}_{\FS}(\frak{c})$, then for $\chi\in\Xi$, $\lambda_{\mathbf{f}\otimes\chi}(\frak{b})= \chi(\frak{b})\lambda_{\mathbf{f}}(\frak{b})$ holds for the Fourier coefficients. Combining this with the above observations, (\ref{average}) is clear. Now
\begin{equation*}
\begin{split}
&\frac{1}{|\Xi|}\sum_{\chi\in\Xi} \sum_{\mathbf{f}\in\mathbf{B}_{\FS}(\frak{c})}\omega_{\mathbf{f}\otimes{\chi}}(\pi_{\frak{m}}^{-1}) (\overline{\eta}\theta\lambda) \lambda_{\mathbf{f}\otimes{\chi}}(\alpha\frak{a}^{-1}) \overline{\lambda_{\mathbf{f}\otimes{\chi}}(\alpha'\frak{a}'^{-1})}=\\ &\frac{1}{|\Xi|} \sum_{\mathbf{f}\in\mathbf{B}_{\FS}(\frak{c})}\omega_{\mathbf{f}}(\pi_{\frak{m}}^{-1}) (\overline{\eta}\theta\lambda) \lambda_{\mathbf{f}}(\alpha\frak{a}^{-1}) \overline{\lambda_{\mathbf{f}}(\alpha'\frak{a}'^{-1})} \sum_{\chi\in\Xi} \chi(\frak{m}^{-2}\frak{a}^{-1}\frak{a}').
\end{split}
\end{equation*}
By definition, the inner sum is $|\Xi|$ if $\frak{m}\in C$, and $0$ if $\frak{m}\in C'$.

The same argument works for $CSC$.
\end{proof}

\begin{lemm} We have the preliminary sum formula
\begin{equation}\label{preliminarysumformula}
\begin{split}
&\left[K(\frak{o}):K(\frak{c})\right]^{-1} \sum_{\mathbf{f}\in\mathbf{B}(\frak{c})} (\overline{\eta}\theta\lambda) \lambda_{\mathbf{f}}(\alpha\frak{a}^{-1}) \overline{\lambda_{\mathbf{f}}(\alpha'\frak{a}'^{-1})}+CSC=\\
&\const\Delta(\alpha\frak{a}^{-1},\alpha'\frak{a}'^{-1}) \int(\overline{\eta}\theta\lambda) d\mu+\\ &\const \sum_{\frak{m}\in C}\sum_{c\in \frak{amc}}\sum_{\epsilon\in\rfact{\frak{o}_+^{\times}}{\frak{o}^{2\times}}} \frac{KS(\epsilon\alpha,\frak{a}^{-1}\frak{d}^{-1};\alpha'\gamma_{\frak{m}}, \frak{a}'^{-1}\frak{d}^{-1};c,\frak{a}^{-1}\frak{m}^{-1}\frak{d}^{-1}) }{\norm{c\frak{a}^{-1}\frak{m}^{-1}}}
\int\mathcal{B}_{(\nu,p)}\left(\frac{|\alpha\alpha'\gamma_\frak{m} \epsilon|^{\frac{1}{2}}}{c}\right) (\overline{\eta}\theta\lambda) d\mu.
\end{split}
\end{equation}
\end{lemm}
\begin{proof} Using that $\mathbf{B}(\frak{c})$ consists of those elements of $\mathbf{B}_{\FS}(\frak{c})$ on which $Z(\AAA)$ acts trivially, we can rewrite the LHS as
\begin{equation*}
\left[K(\frak{o}):K(\frak{c})\right]^{-1} \frac{1}{|C\cup C'|}\sum_{\frak{m}\in C\cup C'} \sum_{\mathbf{f}\in\mathbf{B}_{\FS}(\frak{c})}\omega_{\mathbf{f}}(\pi_{\frak{m}}^{-1}) (\overline{\eta}\theta\lambda) \lambda_{\mathbf{f}}(\alpha\frak{a}^{-1}) \overline{\lambda_{\mathbf{f}}(\alpha'\frak{a}'^{-1})} + CSC.
\end{equation*}
Now the contribution of $\frak{m}\in C$ is given in (\ref{preliminarysumformula'}), while the contribution of $\frak{m}\in C'$ is $0$ by Lemma \ref{harcos}. Note that $|C|$ does not depend on $\frak{a},\frak{a}'$, since $C$ is a coset of the squares in the narrow class group.
\end{proof}

\subsection{Extension of the preliminary sum formula}

Now we are in the position to prove our theorem. Observe that (\ref{preliminarysumformula}) resembles (\ref{kuznetsovformula}), except for the weight function, which is a triple product $\overline{\eta}\theta\lambda$ of functions in the preliminary sum formula, and a single function in the Kuznetsov formula.

\begin{proof}[Proof of Theorem \ref{kuznetsov}]
First choose $q_j>\max(2,a_j)$ at real, $\min(l_j,q_j)>\max(2,a_j)$ at complex places. Choose moreover a small $\delta>0$. Then let $\eta(\nu,p)=e^{\delta\nu^2}$ on $\Re\nu\leq 2/3$ and in the discrete series at real places, let $\eta(\nu)=1$, if $\nu\in 1/2+\ZZ$ and $3/2\leq|\nu|\leq a_j$.

By the assumption $l_j>2$, we see that $\lambda_j\neq 0$ at complex places. At real places, we claim that
\begin{equation*}
\lambda(\nu)=\sum_{\pm}\frac{1}{\Gamma \left(\frac{1}{2}-\nu\pm\frac{q}{2} \right) \Gamma \left(\frac{1}{2}+\nu\pm\frac{q}{2} \right)}\neq 0
\end{equation*}
on the domain
\begin{equation*}
D=\left\{\nu\in\CC:\Re\nu=0\right\} \cup\left(-\frac{1}{2},\frac{1}{2}\right) \cup\left\{\nu\in\frac{1}{2}+\ZZ:|\nu|\leq\frac{q-1}{2}\right\}.
\end{equation*}
Indeed, for $\Im\nu=0$, $3/2\leq q/2+1/2\in 1/2+\ZZ$. This shows that $\Gamma(1/2+q/2+\nu),\Gamma(1/2+q/2-\nu)$ are both positive, so the term '$+q/2$' gives a positive number. Similarly, it is easy to see that in the term '$-q/2$', $\Gamma(1/2-q/2+\nu),\Gamma(1/2-q/2-\nu)$ are either of the same sign or both show a pole (for $\nu\in 1/2+\ZZ$). In any case, they give a non-negative contribution. For $\Re\nu=0$, the positivity is clear, as there are complex norms in the denominators.

Now let $0<\epsilon<1/2$. At real places, we construct functions $\lambda_{\epsilon}(\nu)$ which are holomorphic and non-vanishing on the domain $|\Re\nu|\leq (q-1)/2$ with the further properties
\begin{enumerate}[(i)]
\item $\lambda(\nu)/\lambda_{\epsilon}(\nu)$ is close to $1$ (in terms of $\epsilon$), if $\nu\in D$;
\newline \item $\lambda_{\epsilon}(\nu)$ has polynomial growth along vertical lines.
\end{enumerate}
For example
\begin{equation*}
\lambda_{\epsilon}(\nu)= \frac{\lambda(\nu)}{(\nu-A_1)\cdot\ldots\cdot (\nu-A_u)} B_{\epsilon}(\nu) 
\end{equation*}
is such a function, where $A_i$'s are the zeros of $\lambda$ with multiplicity (there are finitely many of them) and $B_{\epsilon}$ is a holomorphic function with values in the $\epsilon$-neighborhood of $1$ on the domain $|\Re\nu|\leq (q-1)/2$. Indeed, holomorphy and non-vanishing are clear. Moreover, since $A_i\notin D$ for all $i$, and $D$ is closed, (i) is also guaranteed. Finally, (ii) follows from the polynomial growth of $\Gamma$ along vertical lines.

At complex places, let $\lambda_{\epsilon}=\lambda$.

We have already given $\eta$. Let $\theta(\nu,p)=h(\nu,p)e^{-\delta\nu^2}/\lambda_{\epsilon}(\nu,p)$. By construction, this can be chosen to be a test function, if $\delta$ is small enough (independently of $\epsilon$).

Now the triple product gives $h(\nu,p)(\lambda(\nu,p)/\lambda_{\epsilon}(\nu,p)$). Here, the second factor is close to $1$ (in terms of $\epsilon$) on the domain we should integrate it, by the construction of $\lambda_{\epsilon}$. Let $\epsilon\rightarrow 0$, then we obtain the statement by dominated convergence: $2h(\nu,p)$ is an integrable majorant, integrability follows from standard bounds on $J$-Bessel functions in the real case and from the bound \cite[Lemma 9.1.7]{Lokvenec} on $\mathcal{B}_{(\nu,p)}$ in the complex case.
\end{proof}

\appendix
\section{Convergence of Poincar\'e series}

Earlier we put aside the convergence issue in Section \ref{Poincare series}. Now we return to this question. There is a technical difficulty, which does not arise if $F$ has at least one real archimedean embedding. We start with this case, then turn to the harder one, when all the archimedean places are complex.

\subsection{Case 1: $F$ has at least one real embedding ($r\geq 1$)}

We start with a lemma from \cite{BruggemanMiatello2}.
\begin{lemm}\label{sumoverunits}
Let $a,b\in\RR$, $a+b>0$. Assume $f$ is a function on $(\RR^{\times})^r\times (\CC^{\times})^s$ satisfying
\begin{equation*}
f(y)\ll \prod_{j=1}^{r+s}\min(|y_j|^{a\deg[F_j:\RR]},|y_j|^{-b\deg[F_j:\RR]}).
\end{equation*}
Then
\begin{equation*}
\sum_{\epsilon\in\frak{o}^{\times}}f(\epsilon y)\ll (1+|\log|y||^{r+s-1}) \min(|y|^a,|y|^{-b}).
\end{equation*}
\end{lemm}
\begin{proof} See \cite[Lemma 8.1]{BruggemanMiatello2}.
\end{proof}

We focus on $P_1$, the proof is the same for $P_2$. We have the local bounds:
\begin{itemize}
\item for $j\leq r$, by \cite[(3.11)]{BruggemanMiatello},
\begin{equation*}
(\widetilde{\mathcal{L}}_q^{\alpha}\eta)\left(\begin{pmatrix}1 & x \cr \ & 1\end{pmatrix} \begin{pmatrix}y & \ \cr \ & 1\end{pmatrix}\begin{pmatrix}\cos\theta & \sin\theta \cr -\sin\theta & \cos\theta\end{pmatrix} \begin{pmatrix}\pm 1 & \ \cr \ & 1\end{pmatrix}\right)\ll \min(y^{\frac{1}{2}+\sigma},y^{\frac{1}{2}-\sigma}),
\end{equation*}
\item while for $j>r$, by \cite[(9.18)]{Lokvenec},
\begin{equation*}
\widetilde{\mathcal{L}}_{l,q}^{\alpha}\left(\begin{pmatrix}1 & x \cr \ & 1\end{pmatrix} \begin{pmatrix}y & \ \cr \ & 1\end{pmatrix}\begin{pmatrix}\beta & \gamma \cr -\overline{\gamma} & \overline{\beta} \end{pmatrix}\right)\ll \left\{\begin{matrix} y^{1+t}\quad \mathrm{as\ }y\rightarrow 0\mathrm{\ for\ all\ }t\in(0,1),\cr y^{-k}\quad \mathrm{as\ }y\rightarrow \infty\mathrm{\ for\ all\ }k\geq 1.\end{matrix}\right.
\end{equation*}
\end{itemize}
Of course, the implied constants depend on the function $\eta$ and in the complex case, on the chosen numbers $t,k$. Now in the definition (\ref{poincareseries1}) of $P_1$, we may focus on the component $P_1^{\frak{a}}$. Rewrite it as
\begin{equation*}
P_1^{\frak{a}}(g)= \sum_{\gamma\in \lfact{Z_{\Gamma}\Gamma_N(\frak{a},\frak{c})}{\Gamma(\frak{a},\frak{c})}} f_1(\gamma g)= \sum_{ \gamma'\in \lfact{\Gamma_{\infty}(\frak{a},\frak{c})}{\Gamma(\frak{a},\frak{c})}} \sum_{ \gamma\in \lfact{Z_{\Gamma}\Gamma_{N}(\frak{a},\frak{c})}{\Gamma_{\infty}(\frak{a},\frak{c})}} f_1(\gamma\gamma' g)= \sum_{ \gamma'\in \lfact{\Gamma_{\infty}(\frak{a},\frak{c})}{\Gamma(\frak{a},\frak{c})}} h(\gamma'g),
\end{equation*}
where
\begin{equation*}
h(g)=\sum_{ \gamma\in \lfact{Z_{\Gamma}\Gamma_{N}(\frak{a},\frak{c})}{\Gamma_{\infty}(\frak{a},\frak{c})}} f_1(\gamma g),
\end{equation*} and $\Gamma_{\infty}$ stands for the upper triangular subgroup of $\Gamma$. Now observe the coset space $\lfact{Z_{\Gamma}\Gamma_{N}(\frak{a},\frak{c})}{ \Gamma_{\infty}(\frak{a},\frak{c})}$ is covered by the set
\begin{equation*}
\left\{\begin{pmatrix}(\pm)^r\epsilon & \ \cr \ & \epsilon^{-1}\end{pmatrix}: \epsilon\in\frak{o}^{\times}\right\},
\end{equation*}
where by $(\pm)^r$, we mean that at each real place there can be some sign. This $(\pm)^r$ results only a finite, $2^r$ term summation and apart from this, our summation is over the units. Hence Lemma \ref{sumoverunits} applies and gives
\begin{equation*}
h(g)\ll \left\{ \begin{matrix}|y|^{a-\epsilon}\ (|y|\rightarrow 0), \cr |y|^{-b+\epsilon}\ (|y|\rightarrow \infty),\end{matrix}\right.
\end{equation*}
where $y$ is the height of $g$ (the diagonal factor in the Iwasawa decomposition: the quotient of the upper-left and the lower-right entry), $a$ is the product of $1/2+\sigma$'s at the real places and $(1+t)/2$'s at the complex places and $b$ is the product of $\sigma-1/2$'s and $k/2$'s. We can easily guarantee $a>1$ and $b>0$, and for a small $\epsilon>0$, $a-\epsilon>1$.
Now
\begin{equation*}
\sum_{ \gamma'\in \lfact{\Gamma_{\infty}(\frak{a},\frak{c})}{\Gamma(\frak{a},\frak{c})}} h(\gamma'g)\ll \sum_{ \gamma'\in \lfact{\Gamma_{\infty}(\frak{a},\frak{c})}{\Gamma(\frak{a},\frak{c})}} |y(\gamma'g)|^{a-\epsilon},
\end{equation*}
where $y(\gamma' g)$ means the height of $\gamma' g$. On the RHS, we see an Eisenstein series, which is absolutely convergent (as $a-\epsilon>0$). Moreover, on the LHS, the contribution of the upper-triangular element is bounded (as $b>0$), so the resulting function is also bounded, hence square-integrable. (See \cite[Lemma 2.4]{BruggemanMiatello} and \cite[p.648]{BruggemanMiatello2}.)

\subsection{Case 2: $F$ has no real embedding ($r=0$)}

If all the places are complex, we cannot give a similar argument, as the inverse Lebedev transform $\widetilde{\mathcal{L}}_{l,q}^{\alpha}\eta$ does not tend to $0$ fast enough as the height goes to $0$.

In this case we have a little technical simplification, as $\mathrm{PSL}_2(\CC)=\mathrm{PGL}_2(\CC)$, so we can assume that all occuring complex matrices have determinant $1$.

To solve our problem, we follow the argument of Bruggeman and Motohashi \cite[Section 9]{BruggemanMotohashi}. We will also refer to the thesis of Lokvenec-Guleska \cite{Lokvenec}.

Let
\begin{equation*}
B(\eta)=2\pi l\cdot l!\eta(0,1)|\alpha|^2\sqrt{2l+1} {{2l}\choose{l-1}}^{-\frac{1}{2}} {{2l}\choose{l-q}}^{\frac{1}{2}}.
\end{equation*}

Then \cite[(7.14-15)]{BruggemanMotohashi}, \cite[(9.16-17)]{Lokvenec} can be written as
\begin{equation}\label{complexpoincaresplit}
(\widetilde{\mathcal{L}}_{l,q}^{\alpha}\eta)(g)= B(\eta)\M_{\alpha}\varphi_{l,q}(1,0)(g)+O(|y|^{(1+\sigma)/2}),
\end{equation}
where $y$ is the height of $g$. Note that $\sigma>1$, and here we indicated the function $\varphi$ and its weight and spectral data, the latter evaluated at $(1,0)$ and we may assume that $l>0$ (if $l=0$, then this gives $0$, and we have the needed order of magnitude).

Use $\prod\M_{\alpha}$ as the building block (where we dropped $j$ from $\prod_j$), and let (see \cite[(9.1)]{BruggemanMotohashi})
\begin{equation*}
P\prod\M_{\alpha}\varphi_{l,q}(\nu,p)(g)=\sum_{\gamma\in \lfact{Z_{\Gamma}\Gamma_N(\frak{a},\frak{c})}{\Gamma(\frak{a},\frak{c})}} \prod\M_{\alpha}\varphi_{l,q}(\nu,p)(\gamma g).
\end{equation*}
This is absolutely convergent for $\Re\nu>1$ (see \cite[(4.54)]{Lokvenec}, then an Eisenstein series again majorizes our sum). Note that when $\nu$ is a vector in $\CC^s$, by $\Re\nu>1$ we mean $\Re\nu_j>1$ for all $j$. In order to keep notations as simple as it is possible, we will use this abbreviation from now on, not only for $\nu$, but for $p,l,q$ as well.

Now take any building block $f$ which is a pure tensor, and follow \cite[(5.1-6)]{BruggemanMotohashi}. Using the Bruhat decomposition, the Poincar\'e series $Pf$ can be written formally as
\begin{equation*}
\begin{split}
Pf(g)&=\sum_{\epsilon\in\rfact{\frak{o}^{\times}}{\frak{o}^{2\times}}} f\left(\begin{pmatrix} \epsilon & \ \cr \ & \epsilon^{-1} \end{pmatrix}g\right)\\ &+ \sum_{0\neq c\in\frak{ac}} \sum_{d} \sum_{\omega\in(\frak{ad})^{-1}} f\left( \begin{pmatrix} 1 & d'c^{-1} \cr \ & 1\end{pmatrix} \begin{pmatrix} \ & 1 \cr -1 & \ \end{pmatrix} \begin{pmatrix} c & \ \cr \ & c^{-1}\end{pmatrix} \begin{pmatrix} 1 & dc^{-1}+\omega \cr \ & 1\end{pmatrix} g\right),
\end{split}
\end{equation*}
where $d'$ is the element modulo $(\frak{ad})^{-1}c$ such that $dd'\equiv 1$ modulo $(\frak{ad})^{-1}c$ and we sum over those $d$'s for which such a $d'\in\frak{o}$ exists, that is, $d$ generates $\rfact{\frak{o}}{(\frak{ad})^{-1}c}$ as an $\frak{o}$-module. Now applying Poisson summation, we see the $\omega$-sum can be rewritten as
\begin{equation*}
\norm{\frak{ad}}\sum_{\omega\in\frak{ad}}\psi_{\infty}(d\omega/c) \int_{F_{\infty}}\psi_{\infty}(\omega x)^{-1} f\left( \begin{pmatrix} 1 & d'c^{-1} \cr \ & 1\end{pmatrix} \begin{pmatrix} \ & 1 \cr -1 & \ \end{pmatrix} \begin{pmatrix} c & \ \cr \ & c^{-1}\end{pmatrix} \begin{pmatrix} 1 & x \cr \ & 1\end{pmatrix} g\right)dx.
\end{equation*}
Now assume that for some $\omega'\in \frak{a}^{-3}\frak{d}^{-1}$,
\begin{equation*}
f(n(x)g)=\psi_{\infty}(\omega' x)f(g).
\end{equation*}
Therefore we obtain
\begin{equation*}
\begin{split}
Pf(g)&=\sum_{\epsilon\in\rfact{\frak{o}^{\times}}{\frak{o}^{2\times}}} f\left(\begin{pmatrix} \epsilon & \ \cr \ & \epsilon^{-1} \end{pmatrix}g\right)\\ &+ \sum_{\omega\in\frak{ad}} \sum_{0\neq c\in\frak{ac}} KS(\omega,\frak{a}^{-1}\frak{d}^2;\omega',\frak{a}^3;c,\frak{ad}) \J_{\omega}\left(f\left(\begin{pmatrix} 1/c & \ \cr \ & c\end{pmatrix}g\right)\right),
\end{split}
\end{equation*}
noting that by $\J_{\omega}$, we mean the product of the local integrals (we have chosen $f$ to be a pure tensor).

If $\omega$, $\omega'$, $\frak{a}$ are all fixed and only $c$ varies, Weil's bound (see \cite[(13)]{Venkatesh}) gives
\begin{equation}\label{weilsbound}
KS(\omega,\frak{a}^{-1}\frak{d}^2;\omega',\frak{a}^3;c,\frak{ad})\ll |c|^{\frac{1}{2}+\epsilon}.
\end{equation}
Now specializing the above to $\prod\M_{\alpha}$, $\alpha$ taking the place of $\omega'$, we see (still for $\Re\nu>1$)
\begin{equation}\label{mpoincare}
\begin{split}
P\prod\M_{\alpha}\varphi_{l,q}(\nu,p)(g)&= \sum_{\epsilon\in\rfact{\frak{o}^{\times}}{\frak{o}^{2\times}}} \prod\M_{\alpha}\varphi_{l,q}(\nu,p)\left(\begin{pmatrix} \epsilon & \ \cr \ & \epsilon^{-1} \end{pmatrix}g\right)\\ +&\sum_{0\neq c\in\frak{ac}} KS(0,\frak{a}^{-1}\frak{d}^2;\alpha,\frak{a}^3;c,\frak{ad}) \J_{0}\left(\prod\M_{\alpha}\varphi_{l,q}(\nu,p)\left(\begin{pmatrix} 1/c & \ \cr \ & c\end{pmatrix}g\right)\right)\\ +&\sum_{0\neq\omega\in\frak{ad}}\sum_{0\neq c\in\frak{ac}} KS(\omega,\frak{a}^{-1}\frak{d}^2;\alpha,\frak{a}^3;c,\frak{ad}) \J_{\omega}\left(\prod\M_{\alpha}\varphi_{l,q}(\nu,p)\left(\begin{pmatrix} 1/c & \ \cr \ & c\end{pmatrix}g\right)\right).
\end{split}
\end{equation}
Here, the first term analytically continues to $\nu\in\CC^s$. In the second term, we apply Lemma \ref{jacquettransformofgoodmanwallach} together with (\ref{weilsbound}), this continues to $\Re\nu>0$. In the last term, apply Lemma \ref{jacquettransformofgoodmanwallach} again, together with the explicit form of $\J_{\omega}$, this gives
\begin{equation*}
\sum_{0\neq\omega\in\frak{ad}}  \J_{\omega}\varphi_{l,q}(\nu,p)(g) \sum_{0\neq c\in\frak{ac}} \frac{1}{|c|} \left(\frac{c}{|c|}\right)^p KS(\omega,\frak{a}^{-1}\frak{d}^2;\alpha,\frak{a}^3;c,\frak{ad}) \mathcal{J}_{\nu,p}^*\left(\frac{4\pi}{c}\sqrt{\alpha\omega}\right).
\end{equation*}
Here, the inner sum continues analytically to $\Re\nu>1/2$ by (\ref{weilsbound}). The resulting function is of polynomial order in $\omega$, so the $\omega$-sum gives an analytic function (as the $K$-Bessel function appearing in $\J$ has exponential decay at infinity).

We need one more complement: it may happen that $\alpha\notin\frak{a}^{-3}\frak{d}^{-1}$, but this can be easily handled by taking larger (but fixed) fractional ideals.

Now $P\prod\M_{\alpha}$ is a well-defined Poincar\'e series, however, it fails to be square-integrable. Fix some $0<A\in\RR$. Let $\rho$ be a funcion on $\RR_+^s$ such that it is smooth, $\rho(y)=1$ if $|y|\leq A$, and $\rho(y)=0$ if $|y|\geq A+1$. This extends to $\SLtwo{2}{\CC}$ via the Iwasawa decomposition $na(y)k$ by making it independent of $n$ and $k$.

Let $\prod\M'_{\alpha}(g)=\rho(g)\prod\M_{\alpha}(g)$, then we have, for $\Re\nu>1/2$,
\begin{equation*}
P\prod\M'_{\alpha}(g)=P\prod\M_{\alpha}(g)+(\rho(y)-1) \sum_{\epsilon\in\rfact{\frak{o}^{\times}}{\frak{o}^{2\times}}} \prod\M_{\alpha}\varphi_{l,q}(\nu,p)\left(\begin{pmatrix} \epsilon & \ \cr \ & \epsilon^{-1} \end{pmatrix}g\right),
\end{equation*}
if $A$ is large enough, $y=y(g)$ stands for the height of $g=na(y)k$. Now observe that as $|y|\rightarrow\infty$, the magnitude of $P\prod\M'_{\alpha}$ is determined by the second line of (\ref{mpoincare}) and it is $\ll |y|^{1-\Re\nu}$, so it is bounded according to Lemma \ref{jacquettransformofgoodmanwallach} at $\nu=1$, $p=0$.

Now returning to (\ref{complexpoincaresplit}), define $E\eta$ via
\begin{equation*}
(\widetilde{\mathcal{L}}_{l,q}^{\alpha}\eta)(g)= B(\eta)\M'_{\alpha}\varphi_{l,q}(1,0)(g)+E\eta(g).
\end{equation*}
On the product this gives
\begin{equation*}
\prod(\widetilde{\mathcal{L}}_{l,q}^{\alpha}\eta)(g)= \sum_{S\subseteq\{1,\ldots,s\}}\prod_{\{1,\ldots,s\}\setminus S} B(\eta)\M'_{\alpha}\varphi_{l,q}(1,0)(g)\prod_SE\eta(g),
\end{equation*}
where $\prod_S$ means $\prod_{j\in S}$. Now define
\begin{equation}\label{poincareseries2}
P\prod(\widetilde{\mathcal{L}}_{l,q}^{\alpha}\eta)(g)= P\prod\M'_{\alpha}\varphi_{l,q}(1,0)(g)+P\sum_{\emptyset\neq S\subseteq\{1,\ldots,s\}}\left(\prod_{\{1,\ldots,s\}\setminus S} B(\eta)\M'_{\alpha}\varphi_{l,q}(1,0)\prod_SE\eta\right)(g).
\end{equation}
By our construction, $E\eta(g)\ll |y|^{1+\epsilon}$ for some $\epsilon>0$ as $|y|\rightarrow 0$, and $E\eta(g)\ll |y|^{-k}$ for all $k\in\NN$ as $|y|\rightarrow\infty$. Hence the second Poincar\'e series on the RHS is absolutely convergent and gives a bounded function (the argument of the previous section goes through). This, together with the boundedness of $P\prod\M'$ at $\nu=1,p=0$, gives that defining our Poincar\'e series this way, it is bounded, hence square-integrable.

Altogether, in case of $r=0$, we use (\ref{poincareseries2}) as the definition of the Poincar\'e series, where the first term $P\prod\M'\varphi_{l,q}(1,0)$ is understood via analytic continuation as explained above:
\begin{equation}\label{poincareseries2'}
P\prod\M'\varphi_{l,q}(1,0)= \lim_{\nu\rightarrow 1+}P\prod\M'\varphi_{l,q}(\nu,0),
\end{equation}
where inside the limit, there is an absolutely convergent Poincar\'e series for all $\Re\nu>1$. Finally, we also record what we obtain from (\ref{poincareseries2}), (\ref{poincareseries2'}):
\begin{equation}\label{poincareseries3}
P\prod(\widetilde{\mathcal{L}}_{l,q}^{\alpha}\eta)(g)= \lim_{\nu\rightarrow 1+}P\prod\M'_{\alpha}\varphi_{l,q}(\nu,0)(g)+P\sum_{\emptyset\neq S\subseteq\{1,\ldots,s\}}\left(\prod_{\{1,\ldots,s\}\setminus S} B(\eta)\M'_{\alpha}\varphi_{l,q}(1,0)\prod_SE\eta\right)(g).
\end{equation}

\section*{Acknowledgement} I am grateful to Gergely Harcos, my supervisor, for the enlightening discussions over this topic.

\end{document}